\documentclass[a4paper,10pt]{amsart}
\usepackage{amsfonts}
\usepackage{xypic}
\usepackage{amsthm}
\usepackage{amssymb}
\usepackage{amsmath}
\usepackage{enumerate}

\begin{document}
\title{Analytic twists of modular forms}
\author{Alexandre Peyrot}

\maketitle

\def \l {{\lambda}}
\def \a {{\alpha}}
\def \b {{\beta}}
\def \f {{\phi}}
\def \r {{\rho}}
\def \R {{\mathbb R}}
\def \H {{\mathbb H}}
\def \N {{\mathbb N}}
\def \C {{\mathbb C}}
\def \Z {{\mathbb Z}}
\def \F {{\Phi}}
\def \Q {{\mathbb Q}}
\def \e {{\epsilon }}
\def\GL{\ensuremath{\mathop{\textrm{\normalfont GL}}}}
\def\SL{\ensuremath{\mathop{\textrm{\normalfont SL}}}}
\def\Gal{\ensuremath{\mathop{\textrm{\normalfont Gal}}}}
\def\SU{\ensuremath{\mathop{\textrm{\normalfont SU}}}}
\def\SO{\ensuremath{\mathop{\textrm{\normalfont SO}}}}

\newtheorem{prop}{Proposition}
\newtheorem{claim}{Claim}
\newtheorem{lemma}{Lemma}
\newtheorem{thm}{Theorem}
\newtheorem{ktf}{\textit{Kuznetsov Trace Formula}}
\newtheorem{defn}{Definition}

\theoremstyle{definition}
\newtheorem{exmp}{Example}

\theoremstyle{remark}
\newtheorem{rmk}{Remark}

\newcommand{\mods}[1]{\,(\mathrm{mod}\,{#1})}

\begin{abstract}
We investigate non-correlation of Fourier coefficients of Maass forms against a class of real oscillatory functions, in analogy to known results with Frobenius trace functions. We also establish an equidistribution result for twisted horocycles as a consequence of our non-correlation result.
\end{abstract}

\setcounter{tocdepth}{1}
\tableofcontents

\section{Introduction}
In this paper we are interested in sums of Fourier coefficients of $\GL_2$ Maass forms against a certain class of oscillatory functions. The type of oscillatory functions we consider can be thought as archimedean analogs of trace functions studied by Fouvry, Kowalski and Michel in \cite{FKM}. Our main result gives a non-correlation statement between Fourier coefficients of Maass forms against a family of functions, $K_t: \R_{>0}\rightarrow \C$, depending on a large real parameter $t$.

\subsection{Setup}
We let throughout $f$ be a fixed cuspidal Maass Hecke eigenform for $\SL_2(\Z)$, and denote by $1/4 +t_f^2$ the associated eigenvalue of the Laplacian. The form $f$ admits a Fourier expansion
$$f(z)= \sum_{n\not= 0} \r_f(n)|n|^{-1/2}W_{it_f}(4\pi |n| y) e(nx),$$
where $W_\nu$ is a Whittaker function,
$$W_{it} (y) = \frac{e^{-y/2}}{\Gamma\left(\frac{1}{2}+it\right)} \int_0^\infty e^{-x} x^{it-\frac{1}{2}} \left(1+\frac{x}{y}\right)^{it-\frac{1}{2}} \mathrm{d}x.$$
The Fourier coefficients, $\r_f(n)$, are normalized so that by Rankin-Selberg,
\begin{equation}\label{rs}
\sum_{n\leq X} |\r_f(n)|^2 \asymp X.
\end{equation}
We moreover know that the Fourier coefficients oscillate substantially. For example, the following estimate
\begin{equation}\label{addosci}
\sum_{n\leq x} \r_f(n)e(\alpha n) \ll_f x^{1/2 + \epsilon}
\end{equation}
holds for any $\epsilon >0$ uniformly for all $\alpha\in \R$ (see \cite{Iwaspec} theorem 8.1). In order to understand better the oscillatory nature of the Fourier coefficients, we make the following definition.
\begin{defn}
Let $(K(n))_{n\in \N}$ be a bounded sequence of complex numbers. We say that $(K(n))$ does \emph{not} correlate with $(\r_f(n))$ if we have 
$$\sum_{n\leq x} \r_f(n) K(n) \ll_{f,A} x(\log x)^{-A},$$
for all $A\geq 1, x>1$. 
\end{defn}
Non-correlation statements are therefore a way to measure the extent to which the oscillations of a given sequence ``lines up" with the oscillations of the Fourier coefficients. For example, (\ref{addosci}) gives a non-correlation statement for the additive twist $K(n)=e(\alpha n)$ with a power saving of $1/2 - \epsilon$. Another important example of non-correlation arises when $K(n) = \mu(n)$, the M\"obius function, in which case non-correlation is an incarnation of the Prime Number Theorem (see \cite{FouvryGangulyPNT} for a general result combining this and additive twists). Obtaining power saving statements against the M\"obius function would be equivalent to proving a strong zero-free region towards the Riemann Hypothesis for the $L$-function attached to $f$. We give here a final example, which will be the main motivation for our work: let $p$ be a prime number and let $K$ be an isotypic trace function of conductor $p$, then \cite{FKM} gives a non-correlation result for $(K(n))$ with a  power saving of $1/8 - \epsilon$. 

We will study non-correlation against a family of functions $(K_t)_{t\in \R}$,
$$K_t : \R_{>0} \rightarrow \C,$$
where $t$ is a parameter which we will let grow to infinity. 
\begin{defn}\label{anatrace}
A family of smooth functions $(K_t)_{t\in \R}, K_t:\R_{>0} \rightarrow \C$ is called a family of analytic trace functions if there exist real numbers $a<b, b> 0$ and a family of analytic functions $(M_t(s))_{t\in \R}$ in the strip $a<\Re(s)<b$, such that the following conditions hold. 
\begin{enumerate}[1.]
\item The following integral converges for any $a<\sigma<b$,
\begin{equation}\label{mellintrace}
\frac{1}{2\pi i} \int_{(\sigma)} M_t(s) x^{-s} \mathrm{d}s,
\end{equation}
and is equal to $K_t(x)$ for all $x\in \R_{>0}, t\in \R$.
\item There exist constants $c_1,c_2$ depending on the family $(K_t)_{t\in \R}$ , independent of $t$, such that we may write $M_t(\sigma +i\nu) = g_t(\sigma +i\nu) e(f_t(\sigma+i\nu)),$ in such a way that for all $x\in [t,2t]$, the following 
\begin{equation}\label{condg}
g_t^{(j)}(\sigma + i\nu) \ll_j \nu^{\sigma-1/2-j} \hspace{1 cm} \forall j\geq 0,
\end{equation}
holds, as well as the following conditions on $f_t$.
\begin{enumerate}
\item  
Whenever $|\nu|\leq c_1t$ or $|\nu| \geq c_2 t$, we have
\begin{equation}\label{condfone}
\left|f_t'(\sigma + i\nu) - \frac{1}{2\pi} \log(x)\right| \gg 1.
\end{equation}

\item When $c_1 t\leq |\nu| \leq c_2 t$, either (\ref{condfone}) holds, or we have

\begin{equation}\label{condftwo}
f_t''(\sigma+i\nu) \gg \nu^{-1},
\end{equation}
while for all $\epsilon>0,$  $j\geq 0$,
\begin{equation}\label{condf}
f_t^{(j)} (\sigma + i\nu) \ll_{j,\epsilon} \nu^{1+\epsilon-j}.
\end{equation}

\item Finally, we require that
\begin{equation}\label{condad}
f_t''(\sigma+i\nu) - \frac{1}{2\pi \nu} \gg \nu^{-1},
\end{equation}
whenever $c_1 t\leq |\nu| \leq c_2 t$.
\end{enumerate}
\end{enumerate}

\end{defn}

\begin{rmk}
Throughout the paper, we will abuse notation and say that $K_t$ is an analytic trace function when it arises as part of such a family.
\end{rmk}

\begin{rmk}
Conditions (\ref{mellintrace}) - (\ref{condf}) guarantee by means of stationary phase that the integral representation is concentrated around multiplicative character of conductor $t$. Condition (\ref{condad}) ensures that we avoid functions such as $e(x)$, as motivated in Section \ref{sectex}.
\end{rmk}

\begin{rmk}
By the properties of the Mellin transform, we note that if $K_t(x)$ is an analytic trace function, then for any constant $\alpha \in \R_{>0},$ we have that $K_t(\alpha x)$ is also an analytic trace function.
\end{rmk}

\begin{rmk}
We note that in interesting examples, in conjunction with condition (\ref{condfone}), we will also have some stationary points in the region $c_1 t \leq |\nu| \leq c_2 t$, guaranteeing that $||K_t||_\infty \asymp 1$.
\end{rmk}

\begin{rmk}
We note that in practice, we may always ensure that condition (\ref{mellintrace}) holds, by studying $K_t(x) V\left(\frac{x}{t}\right),$ where $V$ is a smooth compactly supported function in $[\frac{1}{2},2]$. In that case, $M_t(s)$ is given by $\int_0^\infty K_t(x)x^{s-1}\mathrm{d}x,$ and the integral in (\ref{mellintrace}) converges absolutely.
\end{rmk}

We give here some examples of analytic trace functions (see Section \ref{sectex} for proofs).  
\begin{exmp}
The normalized $J$-Bessel function of order $t$,
$$F_{it}(x) := t^{1/2} \Gamma\left(\frac{1}{2}+it\right) J_{it}(x),$$
is an analytic trace function. 
\end{exmp}
This should be thought of as an archimedean analog of Kloosterman sums. We now give as a second example that of higher rank Bessel functions as appearing in \cite{HighBessel}, in analogy to hyper-Kloosterman sums. 
\begin{exmp}
For any $n\geq 3$, the $n$-th rank Bessel function of order $t$,
$$J_{n,t} := \frac{t^{\frac{n-1}{2}}}{2\pi in} \int_{(\frac{1}{4})} \Gamma\left(\frac{s-int}{n}\right) \Gamma\left(\frac{s}{n}+ \frac{it}{n-1}\right)^{n-1} e\left(\frac{s}{4}\right) x^{-s}\mathrm{d}s,$$
is an analytic trace function.
\end{exmp}

We will study sums of the shape
$$S(t) := \sum_n \r_f(n) K_t(n) V\left(\frac{n}{t}\right),$$
where $K_t$ is an analytic trace function and $V$ is a smooth function supported in $[1,2]$ and such that $V^{(j)}(x) \ll_j 1$. for convenience we also normalize $V$ so that $\int V(y) \mathrm{d}y =1$. We will show in Section \ref{anak} that any analytic trace function, $K_t$, satisfies $||K_t||_\infty \ll 1$, so that by Cauchy-Schwarz and (\ref{rs}),  we have that
$$S(t) \ll t.$$
Our main result improves on that bound.

\begin{thm}\label{thm1}
Let $K_t:\R \rightarrow \C$ be an analytic trace function. We have
$$S(t) \ll t^{1- 1/8 + \epsilon},$$
where the implicit constant depends only on  $f,\epsilon$ and on $||K_t||_\infty$. 
\end{thm}

\begin{rmk}
For simplicity we have studied the case where $n\asymp t$. We note that for $N\leq t$, one may study similarly
$$Z(N):= \sum_n \r_f(n) K_t(n) V\left(\frac{n}{N}\right).$$
If for $x\asymp N$, conditions (\ref{condfone}) - (\ref{condad}) hold (which is the case in practice), we may show that
$$Z(N) \ll t^{1/2+\epsilon} N^{3/8},$$
which improves on the trivial bound so long as $N \gg t^{4/5+\epsilon}$. 
\end{rmk}

Our bound has an application to the geometric question of equidistribution of horocycle flows with respect to a twisted signed measure. Let us recall that for every continuous compactly supported function $f$ on $\SL_2(\Z)\backslash\H$, we have 
$$\int_0^1 f(x+iy) \mathrm{d}x \rightarrow \mu\left(\hbox{SL}_2(\Z)\backslash\H\right)^{-1}\int_{\SL_2(\Z)\backslash\H} f(z) \mathrm{d}\mu(z),$$
as $y\rightarrow 0$, where $\mu(z)=\frac{\mathrm{d}x\mathrm{d}y}{y^2}$ denotes the hyperbolic measure (see \cite{zaghoro}). In \cite{MR2068968} Str\"{o}mbergsson gives  a similar result by restricting to subsegments of hyperbolic length $y^{-1/2-\delta}$, i.e. that for any $\delta > 0$ and $f$ as above,
$$\frac{1}{\beta-\alpha} \int_\alpha^\beta f(x+iy) \mathrm{d}x \rightarrow \mu\left(\hbox{SL}_2(\Z)\backslash\H\right)^{-1}\int_{\SL_2(\Z)\backslash\H} f(z) \mathrm{d}\mu(z),$$
uniformly as $y \rightarrow 0$ so long as $\beta-\alpha$ remains bigger than $y^{1/2-\delta}$. We use Theorem \ref{thm1} to give the following twisted version of Str\"ombergsson's result, which is analogous to what is proven in \cite{FKM} for horocycles twisted by Frobenius trace functions.
\begin{thm}\label{horothm}
Let $(K_t)_{t\in \R}$ be a family of analytic trace functions. Let $f$ be a Maass form on $\SL_2(\Z)\backslash\H$, and $V$ be a smooth real valued function with compact support in $[\frac{1}{2},\frac{5}{2}]$ such that $V^{(j)}(x) \ll 1$, for all $j\geq 0$. We then have for any $\delta>0$,
$$\frac{1}{\beta-\alpha} \int_\alpha^\beta f(x+iy) K_{1/y}\left(\frac{x}{y}\right) V(x) \mathrm{d}x \rightarrow 0,$$
uniformly as $y \rightarrow 0$ so long as $\beta-\alpha$ remains bigger than $y^{1/8-\delta}$.
\end{thm}

\subsection{Outline of proof of Theorem 1}
We will show in Section \ref{anak} that our definition of analytic trace function implies that we may essentially write 
$$K_t(x) = \frac{1}{2\pi} \int_{\nu \asymp t} g_t(\sigma+i\nu) e(f_t(\sigma+i\nu)) x^{-\sigma-i\nu} \mathrm{d}\nu.$$
Interchanging order of summation and integration, we may therefore write
$$S(t) = \frac{1}{2\pi} \int_{\nu \asymp t} g_t(\sigma+i\nu)e(f_t(\sigma+i\nu)) \sum_{n=1}^\infty \r_f(n)n^{-\sigma-i\nu} V\left(\frac{n}{t}\right) \mathrm{d}\nu.$$
We then adapt the circle method of Munshi, as in \cite{MR3369905}, allowing us to write the inner sum essentially as
$$\frac{1}{K}\int_K^{2K} \sum_{q \asymp Q} \sum_{\substack{a\asymp Q\\ (a,q)=1}}\frac{1}{aq} \sum_{n\asymp t} \r_f(n) n^{iv} e\left(\frac{n\bar{a}}{q} - \frac{nx}{aq}\right) \sum_{m\asymp t} m^{-i(\nu+v)} e\left(-\frac{m\bar{a}}{q} +\frac{mx}{aq}\right) \mathrm{d}v,$$
where $K\leq t$ is a parameter that will ultimately be chosen optimally to be $K=t^{1/2}$, and $Q=(t/K)^{1/2}$. We may now apply Poisson summation to the $m$-sum, and Voronoi summation to the $n$-sum to arrive at the following expression for $S(t)$,
$$ \sum_{n\ll K} \frac{\r_f(n)}{\sqrt{n}} \sum_{q\asymp Q} \sum_{\substack{(m,q)=1\\ 1\leq |m| \ll q}} e\left(\frac{n\bar{m}}{q}\right) \int_{-K}^K\int_{\nu\asymp t}n^{-i\tau/2} g(q,m,\tau,\nu)e(f(q,m,\tau,\nu)) \mathrm{d}\nu \mathrm{d}\tau,$$
where $g$ is a non-oscillatory amplitude function of size $K$ and $f$ is a well understood phase. In particular, we note that (\ref{condad}) implies that $f''(q,m,\tau,\nu) \gg |\nu|^{-1}$, so that we may use second derivative bounds for multivariable integrals and save in the integral. Applying the Cauchy-Schwarz inequality to get rid of the Fourier coefficients, and using the second derivative bound to save $(Kt)^{1/2}$ in the integral, we arrive at
\begin{align*}
S(t) &\ll Kt^{1/4} \left(\sum_{q,q' \asymp Q} \sum_{m,m' \asymp Q}\left(\frac{Q^{-2}}{K^{1/2}} +   \sum_{\substack{n \asymp t\\ n \equiv q\overline{m'} - q'\overline{m} \mod qq'}} \frac{1}{K^{3/2}|n|^{1/2}}\right)\right)^{1/2}\\
&\ll K^{1/4}t^{3/4} + \frac{t}{K^{1/4}},
\end{align*}
which upon taking $K= t^{1/2}$ gives the desired result.

\subsection{Notations}
Throughout the paper, we will let $f(x) \ll g(x),$ $f(x) \gg g(x)$ and $f(x) = O(g(x))$ denote the usual Vinogradov symbols. The notation $f(x) \asymp g(x)$ will be used to mean that both $f(x) \ll g(x)$ and $g(x) \ll f(x)$ hold. Moreover, any subscript in these notations will be taken to mean that the implied constants are allowed to depend on those parameters. The notation $\bar{a} \mods q$ will always be used to denote the multiplicative inverse of $a$ modulo $q$.

\subsection{Acknowledgements}
I would like to thank Philippe Michel for suggesting this problem to me and for the guidance received throughout this project. I am also very grateful for the numerous enlightening conversations with Ian Petrow. This paper benefited from suggestions and comments from Pierre Le Boudec,  Ramon M. Nunes and Paul Nelson.

\section{Stationary phase integrals}
Throughout the paper, we will need several stationary phase lemmas to estimate oscillatory integrals. In particular, we will regularly be faced with a special kind of oscillatory integral which we now define. Let $W$ be any smooth real valued function, with support in $[a,b] \subset (0,\infty)$, and such that $W^{(j)}(x) \ll_{a,b,j} 1$. We then define 
$$W^\dagger(r,s):= \int_0^\infty W(x)e(-rx)x^{s-1}\mathrm{d}x,$$
where $r\in \R$ and $s\in \C$. Munshi gives in \cite{MR3369905} estimations and asymptotics for $W^\dagger$, however we will also need a slightly more precise version of this asymptotic. To this purpose, we quote from \cite{BKY} a version of the stationary lemma. 
\begin{lemma}\label{BKY}
Let $0<\delta<1/10,$ and $X,Y,V, V_1,Q>0, Z:= Q+X+Y+V_1+1$, and assume that
 $$Y\geq Z^{3\delta}, V_1\geq V \geq \frac{QZ^{\delta/2}}{Y^{1/2}}.$$ 
Suppose that $w$ is a smooth function on $\R$ with support on an interval $[a,b]$ of finite length $V_1$, satisfying
$$w^{(j)}(t) \ll_{j} XV^{-j},$$
 for all $j\geq 0$. Suppose that $h$ is a smooth function on $[a,b]$, such that there exists a unique point $t_0$ in the interval such that $h'(t_0)=0,$ and furthermore that 
\begin{align*}
h^{''}(t) \gg \frac{Y}{Q^2}, &\, h^{(j)}(t) \ll_j \frac{Y}{Q^{j}}, & \hbox{ for } j= 1,2,3,\cdots, t \in [a,b].  
\end{align*} 
Then, the integral defined by 
$$I:= \int_{-\infty}^\infty w(t) e^{ih(t)}\mathrm{d}t$$
has an asymptotic expansion of the form 
$$I= \frac{e^{ih(t_0)}}{\sqrt{h^{''}(t_0)}} \sum_{n\leq 3 \delta^{-1}A} p_n(t_0) + O_{A,\delta}( Z^{-A}),$$
and 
\begin{equation}\label{pn}
p_n(t_0):= \frac{\sqrt{2\pi}e^{\pi i/4}}{n!} \left(\frac{i}{2h^{''}(t_0)}\right)^n G^{(2n)}(t_0),
\end{equation}
where $A$ is arbitrary, and 
\begin{equation}\label{G}
G(t):= w(t) e^{iH(t)}; H(t) = h(t)-h(t_0) - \frac{1}{2} h^{''}(t_0) (t-t_0)^2.
\end{equation}
Furthermore, each $p_n$ is a rational function in $h',h'',\cdots,$ satisfying
\begin{equation}\label{pcontrol}
\frac{\mathrm{d}^j}{\mathrm{d}t_0^j}p_n(t_0)\ll_{j,n} X\left(V^{-j}+Q^{-j}\right)\left((V^2Y/Q^2)^{-n}+Y^{-n/3}\right).
\end{equation}
\end{lemma}
We want to extract the first five terms in the asymptotic expansion, in order to have a small enough error term that will be easy to deal with. We therefore compute
$$p_0(t_0)= \sqrt{2\pi}e(1/8) w(t_0),$$
and 
$$G'(t)= w'(t) e^{iH(t)} + iw(t)H'(t)e^{iH(t)} ,$$
$$G''(t)= e^{iH(t)}(w''(t)+2 iw'(t) H'(t)  +iw(t) H''(t) - w(t)H'(t)^2 ).$$
We now see that $H(t_0)=0,$ while 
$$H'(t)= h'(t) -h''(t_0) (t-t_0),$$
and
$$H''(t)= h''(t) - h''(t_0).$$
Hence, we see that also $H'(t_0), H''(t_0) = 0$.  We therefore have
$$p_1(t_0)= \sqrt{2\pi}e(1/8) \frac{i}{2h''(t_0)} w''(t_0).$$
Noting that only the terms that don't contain $H^{(i)}$ for $i=0,1,2$ survive, and that $H^{(j)}(t)=h^{(j)}(t)$ for $j\geq 3$, we have
\begin{align*}
G^{(4)}(t_0)&= w^{(4)}(t_0) + 4i w'(t_0)h^{(3)}(t_0) +iw(t_0)h^{(4)}(t_0),
\end{align*}
and thus
$$p_2(t_0) = - \frac{\sqrt{2\pi}e\left(\frac{1}{8}\right)}{8 h''(t_0)^2} (w^{(4)}(t_0) + 4i w'(t_0)h^{(3)}(t_0) +iw(t_0)h^{(4)}(t_0)).$$
In general, $G^{(2n)}(t_0)$ is a linear combination of terms of the form 
$$w^{(\nu_0)}(t_0)H^{(\nu_1)}(t_0)\cdots H^{(\nu_l)},$$
where $\nu_0 +\cdots + \nu_l = 2n$.\\
We now wish to use these in the context of the study of $W^\dagger(r,s),$ where we write $s= \sigma + i \beta \in \C$. We may thus use the lemma above with 
$$w(x)= W(x)x^{\sigma-1},$$
and
$$h(x) = -2\pi rx + \beta\log x.$$
Then, 
\begin{equation}\label{sech}
h'(x)= -2\pi r + \frac{\beta}{x}, \, \hbox{ and } h^{(j)}(x) = (-1)^{j-1}(j-1)!\frac{\beta}{x^j},
\end{equation}
for $j\geq 2$. The unique stationary point is given by 
$$x_0 = \frac{\beta}{2\pi r}.$$
We now let
$$\check{W}(x) := x^{1-\sigma} \sum_{n=0}^5 p_n(x),$$
and claim it is non-oscillatory in the following sense.
\begin{claim}
Let $\beta \gg 1$. Then for all $j\geq 0$, and $x\in [a,b],$
$$\check{W}^{j}(x) \ll_{\sigma,j,a,b} 1.$$
\end{claim}
\begin{proof}
We compute
$$\check{W}^{(j)}(x) = \sum_{l=0}^j \left(\begin{array}{l} j\\l\end{array}\right) (x^{1-\sigma})^{(j-l)} \sum_{n=0}^5p_n^{(l)}(x).$$
Now, it is clear that $(x^{1-\sigma})^{(j-l)} \ll_{j, \sigma,a,b}1,$ and so we just need to control the derivatives of each $p_n$. Since $w$ is a product of a power of $x$ with $W$ and $W^{(j)}(x)\ll_j 1$, we can easily see that $p_0(x) \ll_{j,\sigma,a,b} 1$. Now 
$$h''(x_0)= -\frac{\beta}{x^j},$$
and since $\beta \gg 1,$ by the same argument as for $p_0$, it is clear that $p_1(x) \ll 1$. We may apply the same reasoning for $p_2$, and more generally for any $p_n$, since (\ref{sech}) implies the higher derivatives of $h$ don't grow compared to the powers of $h''$ in the denominator.
\end{proof}

We may now give the following result for $W^\dagger(r,s)$.
\begin{lemma}\label{newstat}
Let $r\in \R$ and $s=\sigma + i\beta \in \C$, such that $x_0= \frac{\beta}{2\pi r} \in[a/2, 2b]$. Then,
\begin{align*}
W^\dagger(r,s) =& \frac{\sqrt{2\pi} e(1/8)}{\sqrt{-\beta}} \left(\frac{\beta}{2\pi r}\right)^\sigma \left(\frac{\beta}{2\pi er}\right)^{i\beta} \check{W}\left(\frac{\beta}{2\pi r}\right)+O(\min\{|\beta|^{-5/2},|r|^{-5/2}\}). 
\end{align*}
\end{lemma}
\begin{proof}
This is a direct application of Lemma \ref{BKY} with $X=V=Q=1, Y=\max\{|\beta|, |r|\}, V_1=b-a,$ using the above computations as well as (\ref{pcontrol}).  
\end{proof}
We also quote from \cite{MR3369905} the following lemma.
\begin{lemma}\label{Munshilemma}
$$W^\dagger(r,s) = O_{a,b,\sigma,j} \left(\min\left\{\left(\frac{1+|\beta|}{|r|}\right)^j, \left(\frac{1+|r|}{|b|}\right)^j\right\}\right).$$
\end{lemma}

\section{Analysis of $K_t$}\label{anak} 
In this section, we analyse further the integral representation of $K_t$. We make a partition of unity in the integral: let $\mathcal{I}= \{0\} \cup_{j\geq 0} \{\pm\left(\frac{4}{3}\right)^j\}$, such that  for each $l\in \mathcal{I}$, we take a smooth function $W_l(x)$ supported in $[\frac{3l}{4},\frac{4l}{3}]$ for $l\not= 0$ and such that 
$$x^kW_l^{(k)}(x)\ll_k 1,$$
for all $k\geq 0$. for $l=0$, take $W_0(x)$ supported in $[-2,2]$ with $W_0^{(k)}(x)\ll_l 1$. and such that $1= \sum_{l\in \mathcal{I}} W_l(x).$ We then let for any $i\in \mathcal{I}$, 
$$I_{l,t}(x) :=\frac{1}{2\pi} \int_\R g_t(\sigma + i\nu) e(f_t(\sigma + i\nu)) x^{-\sigma -i\nu} W_l(\nu) \mathrm{d}\nu.$$ 
We prove the following result.
\begin{lemma}\label{locK}
Let $K_t$ be an analytic trace function. We have, for $x\in [t,2t] $, and any $\epsilon >0$,
$$K_t(x) = \sum_{\hbox{Supp}(W_l)\subset [\pm t^{1-\epsilon},\pm t^{1+\epsilon}] \cup [-t^\epsilon,t^\epsilon]} I_{l,t}(x) + O(t^{-1000}).$$
Moreover, we also have
$$\max_{x \in [t,2t]} |K_t(x)| \ll 1.$$
\end{lemma}
\begin{proof}
Condition (\ref{mellintrace}) implies that we may write
\begin{equation}\label{spectralexp}
K_t(x) = \frac{1}{2\pi i } \int_{(\sigma)} M_t(s) x^{-s} \mathrm{d}s= \frac{1}{2\pi} \int_\R g_t(\sigma + i\nu)e(f_t(\sigma+i\nu)) x^{-\sigma-i\nu} \mathrm{d}\nu,
\end{equation}
for any $\sigma \in [a,b]$. We now wish to run a stationary phase argument to localise the integral around the points without too much oscillation. 
If $l\ll t^\epsilon$ for some small $0< \epsilon < \sigma/(1/2+\sigma)$, then 
$$I_{i,t} (x) \ll t^{\epsilon +\epsilon (\sigma-1/2)-\sigma} = o(1),$$
as long as we take $\sigma >0$. We now fix such an $\epsilon$ and look at $l$ such that Supp$(W_l)\subset [\pm t^\epsilon, \pm \infty)$, and look at 
$$x^\sigma I_{l,t}(x) = \int_\R g_t(\sigma +i \nu) W_l(\nu) e\left(f_t(\sigma + i\nu) - \frac{\nu}{2\pi}\log(x)\right) \mathrm{d}\nu,$$
for $x\in [t,2t]$. We now compute a few derivatives, in order to apply stationary phase arguments. We have by (\ref{condg})
$$(g_t(\sigma+i\nu)W_l(\nu))^{(j)}(\nu) \ll_j i^{\sigma-1/2-j}, \hspace{1 cm} \forall j\geq 0,$$
while by (\ref{condfone})
$$f_t'(\sigma+i\nu) - \frac{\log(x)}{2\pi} \gg 1,$$
if $\nu \not\asymp t$ and by (\ref{condf}) 
$$f_t^{(j)}(\sigma +i \nu) \ll l^{1+\epsilon/2 -j}.$$ 
Therefore, in the case that $\nu \not\asymp t$, we may use Lemma \ref{BKY} (with $X= l^{\sigma -1/2}, U=l, \beta-\alpha= 3l/2, R=1, Y=l^{1+\epsilon/2}$ and $Q= l$), to deduce that 
$$I_{l,t}(x) \ll_A l^{-A},$$
for any $A>0$.\\

In the case that $\nu \asymp t$, we use the second derivative bound for oscillatory integrals along with (\ref{condftwo}) to deduce that 
$$I_{l,t}(x) \ll 1.$$
\end{proof}
To conclude this section we note that the case where Supp$(W_l) \subset [-t^\epsilon,t^\epsilon]$ can be handled as follows. Since $V$ is a smooth compactly supported function, it admits a Mellin transform,
$$\tilde{V}(s) = \int_0^\infty V(x) x^{s-1} \mathrm{d}x,$$
that decays very rapidly in vertical strips. One can thus write for any $\alpha \in \R,$
$$V(x) = \int_{(\alpha)} \tilde{V}(s) x^{-s} \mathrm{d}s.$$
Using this, we write for any $\sigma \geq 0$, 
\begin{align*}
\sum_{n=1}^\infty \r_f(n) I_{l,t}(n) V\left(\frac{n}{t}\right) &= \int_\R M_t(\sigma + i\nu)  W_l(\nu) \sum_{n=1}^\infty \r_f(n) n^{-\sigma-i\nu} V\left(\frac{n}{t}\right) \mathrm{d}\nu \\
 &= \int_\R \int_{(\alpha)} M_t(\sigma+i\nu) W_l(\nu) \tilde{V}(s) t^s L(f, \sigma+i\nu +s) \mathrm{d}s\mathrm{d}\nu\\
&\ll t^{1/2+\epsilon},
\end{align*}
by the rapid decay of $\tilde{V}$.\par
 We will therefore only focus on the cases where the support of $W_l$ is close to $t$. This may be interpreted as the fact that the spectral decomposition of any analytic trace function, $K_t$, concentrates around multiplicative characters of conductor $t$.

\section{Proof of Theorem 1}
Following Munshi \cite{MR3369905} we adapt Kloosterman's version of the circle method along with a conductor dropping mechanism. We quote here the following proposition in \cite{IwaniecKowalski}.

\begin{prop}
Let 
$$\delta(n) = \left\{\begin{array}{ll} 1 & \hbox{ if } n=0;\\
0 & \hbox{ otherwise.}\end{array}\right.$$
Then, for any real number $Q\geq 1$, we have
$$\delta(n) = 2\Re \int_0^1 \sideset{}{^*}\sum_{1\leq q\leq Q<a\leq q+Q} \frac{1}{aq} e\left(\frac{n\bar{a}}{q} - \frac{nx}{aq}\right) \mathrm{d}x.$$
\end{prop}
In particular, we will use this proposition with $Q:=(t/K)^{1/2},$ where $t^{\epsilon'}<K<t^{1-\epsilon'}$ (for some $\epsilon'>0$) is a parameter to be chosen optimally later. We let
$$S_l(t) : = \sum_{n=1}^\infty \r_f(n) I_{l,t}(n) V\left(\frac{n}{t}\right),$$
and note that in order to bound non-trivially $S(t)$, it is sufficient to do so for $S_{l}(t)$, for $i$ such that Supp$W_l \subset [\pm t^{1-\epsilon}, \pm t^{1+\epsilon}]$, as follows from the previous section.
 We may thus write
\begin{align*}
S_l(t) & = \sum_{n=1}^\infty \r_f(n) I_{l,t}(n) V\left(\frac{n}{t}\right)\\
&= \frac{1}{K} \int_\R V\left(\frac{v}{K}\right) \sum_{\substack{n,m=1\ n=m}}^\infty \r_f(n) I_{l,t}(m) \left(\frac{n}{m}\right)^{iv} V\left(\frac{n}{t}\right) U\left(\frac{m}{t}\right) \mathrm{d}v\\
&= S_l^+(t) + S_l^-(t),
\end{align*}
where $U$ is a smooth functions supported in $[1/2,5/2]$, with $U(x) = 1 $ for $x\in$ Supp$(V)$ and $U^{(j)} \ll_j 1$, and 
\small
\begin{align*}
S_l^\pm(t) = & \frac{1}{K} \int_0^1 \int_\R V\left(\frac{v}{K}\right) \sideset{}{^*}\sum_{1\leq q \leq Q <a \leq Q+q} \frac{1}{aq}\\
&\times \sum_{n,m=1}^\infty \r_f(n)n^{iv} I_{l,t}(m) m^{-iv} e\left(\pm \frac{(n-m)\bar{a}}{q} \mp \frac{(n-m)x}{aq}\right) V\left(\frac{n}{t}\right)U\left(\frac{m}{t}\right) \mathrm{d}v\mathrm{d}x.
\end{align*} 
\normalsize
We will now describe the analysis for $S_l^+(t)$ (the analysis for $S_l^-(t)$ being completely analogous). 

\subsection{Summation formulae}
We start with the $m$-sum, which we split into congruence classes mod $q$, and after applying Poisson summation, we obtain
\begin{align*}
&\sum_{m=1}^\infty I_{l,t}(m) m^{-iv} U\left(\frac{m}{t}\right) e\left(-\frac{m\bar{a}}{q}\right) e\left(\frac{mx}{aq}\right) \\
%&= \sum_{\alpha \mod q} e\left(-\frac{\alpha \bar{a}}{q}\right) \sum_{m=1}^\infty I_{i,t}(\alpha +mq) (mq+\alpha)^{-iv} U\left(\frac{\alpha +mq}{t}\right) e\left(\frac{(mq+\alpha)x}{aq}\right) \\
%&= \sum_{\alpha \mod q} e\left(-\frac{\alpha \bar{a}}{q}\right) \sum_{m\in \Z} \int_\R I_{i,t}(\alpha +yq) (yq+\alpha)^{-iv}  U\left(\frac{\alpha +yq}{t}\right) e\left(\frac{(yq+\alpha)x}{aq}-my\right) \mathrm{d}y\\
%&= \sum_{\alpha \mod q} e\left(-\frac{\alpha \bar{a}}{q}\right) \sum_{m \in \Z} \frac{t}{q} \int_\R I_{i,t}(tu) (tu)^{-iv} U(u) e\left(\frac{tux}{aq}\right) e\left(-\frac{mtu}{q}\right) e\left(\frac{m\alpha}{q}\right)  \mathrm{d}u\\
%&= \sum_{m\in \Z} \left(\sum_{\alpha \mod q} e\left(-\frac{\alpha \bar{a}}{q}+\frac{m\alpha}{q}\right)\right) \frac{t^{1-iv}}{q} \int_\R I_{i,t}(tu) e\left(\frac{tu(x-ma)}{aq}\right) U(u) u^{-iv} \mathrm{d}u\\
%&= \frac{t^{1-iv}}{2\pi} \sum_{\substack{m\in \Z \\ m \equiv \bar{a} \mod q}} \int_\R M_t(\sigma + i\nu) W_i(\nu) \int_\R(tu)^{-\sigma-i\nu} u^{-iv} U(u) e\left(\frac{tu(x-ma)}{aq}\right) \mathrm{d}u\mathrm{d}\nu\\
&= \sum_{\substack{m\in \Z \\ m\equiv \bar{a} \mods q}} \frac{t^{1-\sigma-iv}}{2\pi} \int_\R t^{-i\nu} M_t(\sigma + i\nu) W_l(\nu) U^\dagger\left(\frac{t(ma-x)}{aq},1-\sigma - i(\nu +v)\right) \mathrm{d}\nu.
\end{align*}
We now note that  since $|\nu| \in [t^{1-\epsilon}, t^{1+\epsilon}]$, we may as in \cite{MR3369905} use Lemma \ref{Munshilemma} to deduce that only the contribution from  $1\leq |m|\ll qt^\epsilon$ is non-negligible. We take a dyadic subdivision to obtain the following.

\begin{lemma}
$$S_l^+(t)= \frac{t^{1-\sigma}}{K} \sum_{1\leq C\leq (t/K)^{1/2}} S_l(t,C) + O(t^{-1000}),$$
where $C$ runs over dyadic integers and
\begin{align*}
&S_l(t,C)= \frac{1}{2\pi} \int_\R \int_0^1 \int_\R M_t(\sigma+i\nu) W_l(\nu) t^{-i(v+\nu)} V\left(\frac{v}{K}\right) \sum_{C<q\leq 2C} \sum_{\substack{(m,q)=1 \\ 1\leq |m| \ll qt^\epsilon}} \frac{1}{aq} \\
&\times U^\dagger\left(\frac{t(ma-x)}{aq},1-\sigma-i(v+\nu)\right)  \sum_{n=1}^\infty \r_f(n) n^{iv} e\left(\frac{nm}{q}\right) e\left(-\frac{nx}{aq}\right) V\left(\frac{n}{t}\right) \mathrm{d}v\mathrm{d}x\mathrm{d}\nu
\end{align*} 
and $a=a_Q(m,q)$ is the unique multiplicative inverse of $m \mod q$ in $(Q,q+Q]$.
\end{lemma}
We wish to use the Voronoi summation on the $n$-sum. We quote from \cite{MR1915038}  the following formula. 
\begin{lemma}
Let $g$ be a Hecke-Maass form over $\SL_2(\Z)$ and spectral parameter $t_g$. Let $F$ be a smooth function decaying at infinity, which vanishes in a neighborhood of the origin. Then, for $(a,c)=1$, we have
\begin{align*}
\sum_{n\geq 1} \r_g(n) e\left(\frac{an}{c}\right) F(n) 
%& = \frac{1}{c}\sum_{n\geq 1} \r_g(n) e\left(-\frac{n\bar{a}}{c}\right) \int_0^\infty F(x) J_g\left(\frac{4\pi\sqrt{nx}}{c}\right)\mathrm{d}x\\
%&+\frac{1}{c}\sum_{n\geq 1}\r_g(-n) e\left(\frac{n\bar{a}}{c}\right)\int_0^\infty F(x) K_g\left(\frac{4\pi\sqrt{nx}}{c}\right)\mathrm{d}x,\\
&=\frac{1}{c}\sum_{\pm}\sum_{n\geq 1} \r_f(\mp n) e\left(\pm \frac{n\bar{a}}{c}\right)V^\pm\left(\frac{n}{c^2}\right),
\end{align*}
where
\begin{align*}
V^-(y)&= \int_0^\infty F(x) J_g(4\pi\sqrt{xy}) \mathrm{d}x\\
V^+(y) &= \int_0^\infty F(x) K_g(4\pi\sqrt{xy})\mathrm{d}x,
\end{align*}
and
\begin{align*}
J_g(x)&= -\frac{\pi}{\sin(\pi it_g)}\left(J_{2it_g}(x)-J_{-2it_g}(x)\right),
\end{align*}
and
\begin{align*}
K_g(x)&= 4\cos(\pi it_g) K_{2it_g}(x).
\end{align*}

\end{lemma}

We now use \cite[p. 326, 331]{tablestransf} that
\begin{align*}
K_{2ir}(x)&=\frac{1}{4}\frac{1}{2\pi i} \int_{(\sigma')} \left(\frac{x}{2}\right)^{-s} \Gamma\left(\frac{s}{2}+ir\right)\Gamma\left(\frac{s}{2}-ir\right) \mathrm{d}s, & |\Re(2ir)| < \sigma'\\
J_{2ir}(x)&=\frac{1}{2}\frac{1}{2\pi i} \int_{(\sigma')} \left(\frac{x}{2}\right)^{-s} \frac{\Gamma(s/2+ir)}{\Gamma(1-s/2+ir)}\mathrm{d}s,& -\Re(2ir) < \sigma' <1,
\end{align*}
and define
\begin{align*}
\gamma_-(s)&= \frac{-\pi}{4\pi i \sin(\pi it_g)} \left\{\frac{\Gamma(s/2 +i t_g)}{\Gamma(1-s/2 + i t_g)} - \frac{\Gamma(s/2- i t_g)}{\Gamma(1-s/2-it_g)}\right\}\\
\gamma_+(s)&= \frac{4\cos(\pi i t_g)}{8\pi i} \Gamma\left(\frac{s}{2}+i t_g\right)\Gamma\left(\frac{s}{2} - it_g\right) 
\end{align*}
to deduce that for any $0<\sigma' <1$,
$$V^-(y)=  \int_0^\infty F(x) \int_{(\sigma')} (2\pi \sqrt{xy})^{-s} \gamma_-(s) \mathrm{d}s \mathrm{d}x,$$
and
$$V^+(y)=  \int_0^\infty F(x) \int_{(\sigma')} (2\pi \sqrt{xy})^{-s} \gamma_+(s)  \mathrm{d}s \mathrm{d}x.$$
We are now ready to plug all of this into effect. 
\small
\begin{align*}
&\sum_{n\geq 1} \r_f(n) e\left(\frac{nm}{q}\right) n^{iv} e\left(\frac{-nx}{aq}\right) V\left(\frac{n}{t}\right)%\\
%&= \frac{1}{q} \sum_\pm \sum_{n\geq 1} \r_f(\mp n) e\left(\pm \frac{n\bar{m}}{q}\right)  \int_{(\sigma')} \left(\frac{2\pi \sqrt{n}}{q}\right)^{-s} \gamma_\pm(s) \int_0^\infty y^{iv} e\left(\frac{-yx}{aq}\right) V\left(\frac{y}{t}\right) y^{-s/2} \mathrm{d}y\mathrm{d}s \\
&= \frac{t^{1+iv}}{q} \sum_\pm \sum_{n\geq 1} \r_f(\mp n) e\left(\pm \frac{n\bar{m}}{q}\right) I(n,q,v,x),
\end{align*}
\normalsize
where 
\begin{align*}
I(n,q,v,x)& = \int_{(\sigma')} \left(\frac{2\pi\sqrt{nt}}{q}\right)^{-s} \gamma_\pm(s) \int_0^\infty y^{iv}e\left(\frac{-tyx}{aq}\right)V(y) y^{-s/2} \mathrm{d}y\mathrm{d}s\\
&= \int_{(\sigma')} \left(\frac{2\pi  \sqrt{nt}}{q}\right)^{-s} \gamma_\pm(s) V^\dagger\left(\frac{tx}{aq},1+iv-s/2\right)\mathrm{d}s.
\end{align*}
By Stirling's formula:
\begin{align*}
&\Gamma(\sigma' + it)\\
&= \sqrt{2\pi} \exp\left(\frac{-\pi |t|}{2}\right)|t|^{\sigma' - 1/2} \left|\frac{t}{e}\right|^{it} \exp(\hbox{sign}(t)i\pi(\sigma'-1/2)/2)(1+O(|t|^{-1})),
\end{align*}
for $|t| \geq 1$ and bounded $\sigma'$, we deduce that
$$\gamma_\pm(\sigma' + i\tau) \ll 1 + |\tau|^{\sigma'-1}. $$
Now, by Lemma \ref{Munshilemma}, 
$$V^\dagger\left(\frac{tx}{aq},1+iv-s/2\right) \ll \min\left\{1, \left(\frac{ (Kt)^{1/2}}{|v-\tau/2| q}\right)^j\right\}.$$
Thus, shifting the contour to $\sigma'=M$ a large positive integer and taking $j=  M +1$ for instance, we see that if $n\gg K t^\epsilon$, then
the integral is negligible (by splitting the integral into a box around $|v-\frac{\tau}{2}|q\leq (Kt)^{1/2}$ and its complement). In the remaining range, we study this more closely. We shift our contour to $\sigma=1$ (the $\gamma_+$ contribution is trivial, so we only consider $\gamma_-$), and note that \begin{align*}
\gamma_-(i\tau+1)&= \left(\frac{|\tau|}{2e}\right)^{i\tau} \Phi_-(\tau),
\end{align*}
where $\Phi_-'(\tau) \ll |\tau|^{-1}.$ We thus have
\small
\begin{align*}
I(n,q,v,x)&= \frac{qi}{2\pi\sqrt{nt}} \sum_{J\in \mathcal{J}} \int_\R \left(\frac{2\pi \sqrt{nt}}{q}\right)^{-i\tau} \gamma_\pm(i\tau+1) V^\dagger\left(\frac{tx}{aq},\frac{1}{2}+i(v-\tau/2)\right) W_J(\tau)\mathrm{d}\tau \\& + O(t^{-1000}),
\end{align*}
\normalsize
where $\mathcal{J}$ is a collection of $O(\log t)$ integers such that $J\in \mathcal{J}$ if and only if Supp$W_J \subset [-(tK)^{1/2}t^\epsilon/C,(tK)^{1/2}t^\epsilon/C]$. We have proven the following:

\begin{lemma}\label{lemmastc}
\small
\begin{align*}
S_l(t,C)&= \frac{iKt^{1/2}}{4\pi^2} \sum_\pm \sum_{J\in \mathcal{J}} \sum_{n\ll Kt^\epsilon} \frac{\r_f(\mp n)}{\sqrt{n}} \sum_{C<q\leq 2C} \sum_{\substack{(m,q)=1\\ 1\leq |m| \ll qt^\epsilon}} \frac{e\left(\pm \frac{n\bar{m}}{q}\right)}{aq} I^*_\pm(q,m,n)\\
& + O(t^{-10000}),
\end{align*}
\normalsize
where
\small
\begin{align*}
I_\pm^*(q,m,n) &= \int_{\R^2} M_t(\sigma+i\nu)W_l(\nu) t^{-i\nu} \left(\frac{2\pi \sqrt{nt}}{q}\right)^{-i\tau} \gamma_\pm (i\tau+1) I^{**}(q,m,\tau,\nu) W_J(\tau) \mathrm{d}\tau \mathrm{d}\nu,
\end{align*}
\normalsize
and 
\small
\begin{align*}
I^{**}(q,m,\tau,\nu) &= \int_0^1 \int_\R V(v) V^\dagger\left(\frac{tx}{aq},i\left(kv-\frac{\tau}{2}\right)+\frac{1}{2}\right)\\
&\hspace{1 cm} \times U^\dagger\left(\frac{t(ma-x)}{aq},1-\sigma-i(Kv+\nu)\right) \mathrm{d}v\mathrm{d}x.
\end{align*}
\end{lemma}
\normalsize
In the next two subsections we evaluate $I^{**}(q,m,\tau,\nu)$.
\subsection{Analysis of the integrals}
We apply lemma \ref{newstat} to
\begin{align*}
&U^\dagger\left(\frac{t(ma-x)}{aq},1-\sigma-i(Kv+\nu)\right)\\&=e\left(\frac{1}{8}\right)\left(\frac{Kv+\nu}{2\pi}\right)^{1/2-\sigma} \left(\frac{aq}{t(x-ma)}\right)^{1-\sigma} \left(\frac{(Kv+\nu)aq}{2\pi et(x-ma)}\right)^{-i(Kv+\nu)} \\
&\hspace{5 cm}\times \check{U}\left(\frac{(Kv+\nu)aq}{2\pi t(x-ma)}\right) + O(t^{-5/2}).
\end{align*}
Hence,
\small
\begin{align*}
I^{**}(q,m,\tau,\nu) =& c_1 \int_0^1 \int_\R V(v) V^\dagger\left(\frac{tx}{aq},i\left(Kv-\frac{\tau}{2}\right) +\frac{1}{2}\right) (Kv+\nu)^{1/2-\sigma} \left(\frac{aq}{t(x-ma)}\right)^{1-\sigma} \\
&\times \left(\frac{(Kv+\nu)aq}{2\pi et(x-ma)}\right)^{-i(Kv+\nu)} \check{U}\left(\frac{(Kv+\nu)aq}{2\pi t(x-ma)}\right) \mathrm{d}v\mathrm{d}x+O(t^{-5/2}), 
\end{align*}
\normalsize
for some constant $c_1$. We now use lemma 5 of \cite{MR3369905} to
\small
 \begin{align*}
V^\dagger\left(\frac{tx}{aq}, i(Kv-\tau/2)+\frac{1}{2}\right) & = \frac{(aq)^{1/2} e(-\frac{1}{8})}{(tx)^{1/2}}  \left(\frac{(Kv-\frac{\tau}{2})aq}{2e\pi tx}\right)^{i(Kv-\frac{\tau}{2})} V\left(\frac{(Kv-\tau/2)aq}{2\pi tx}\right)\\
& + O\left(\min\left\{|Kv-\tau/2|^{-3/2},\left(\frac{tx}{aq}\right)^{-3/2}\right\}\right). 
\end{align*}
\normalsize
Hence,
\small
\begin{align*}
I^{**}(q,m,\tau,\nu) = &c_2 \int_0^1\int_\R V(v) V\left(\frac{(Kv-\frac{\tau}{2})aq}{2\pi tx}\right) \left(\frac{(Kv-\frac{\tau}{2})aq}{2e\pi tx}\right)^{i(Kv-\frac{\tau}{2})} \check{U}\left(\frac{(kv+\nu)aq}{2\pi t(x-ma)}\right)\\
& \left(\frac{aq}{t}\right)^{\frac{3}{2}-\sigma} \frac{(\nu+Kv)^{\frac{1}{2}-\sigma}}{(x-ma)^{1-\sigma}} \left(\frac{(Kv+\nu)aq}{2\pi et(x-ma)}\right)^{-i(Kv+\nu)} \mathrm{d}v\frac{\mathrm{d}x}{x^{1/2}} +E+O(t^{-\frac{5}{2}}),
\end{align*}
\normalsize
for some constant $c_2$ and where $E$ comes from the error term of $V^\dagger$ which we will now describe. We first note that since $V^\dagger \left(\frac{tx}{aq}, i(Kv-\tau/2)+\frac{1}{2}\right)$ does not depend on $\nu$, neither does the error term, and therefore we may perform the $\nu$-integral without losing control of the phase, before plugging absolute values. We thus estimate
\small
\begin{align*}
&\int_\R M_t(\sigma+i\nu) W_l(\nu) t^{-i\nu} (Kv+\nu)^{1/2-\sigma} \left(\frac{(Kv+\nu)aq}{2\pi et(x-ma)}\right)^{-i(Kv+\nu)} \check{U}\left(\frac{(Kv+\nu)aq}{2\pi t(x-ma)}\right) \mathrm{d}\nu\\
&= \int_\R g(\nu)e(f(\nu))\mathrm{d}\nu,
\end{align*}
\normalsize
where, temporarily, we define
$$g(\nu)=g_t(\sigma +i\nu) W_l(\nu) (Kv+\nu)^{1/2-\sigma} \check{U}\left(\frac{(Kv+\nu)aq}{2\pi t(x-ma)}\right),$$
and
$$2\pi f(\nu)= 2\pi f_t(\sigma+i\nu) - \nu\log t - (Kv+\nu) \log\left|\frac{(Kv+\nu)aq}{2\pi et(x-ma)}\right|.$$
We have 
$$2\pi f''(\nu) = 2\pi f_t''(\sigma +i\nu) - \frac{1}{Kv+\nu} \gg \nu^{-1},$$
by (\ref{condad}). Noting that  $g(\nu) \ll  1,$ and $\int |g'(\nu)| \ll t^\epsilon$, we may use the second derivative bound for oscillatory integrals (see \cite{Srinstat}, Lemma 5) to deduce that 
\begin{equation}\label{extint}
\int_\R g(\nu)e(f(\nu)) \mathrm{d}\nu \ll t^{1/2+\epsilon}.
\end{equation}
Our error term, $E$, therefore satisfies 
\begin{align*}
&\int_\R M_t(\sigma +i\nu) W_l(\nu) t^{-i\nu} E \mathrm{d}\nu\\
&\ll t^{\sigma-1/2+\epsilon} \int_0^1\int_1^2 \min\left\{ \left|Kv-\frac{\tau}{2}\right|^{-3/2}, \left(\frac{tx}{aq}\right)^{-3/2}\right\} \mathrm{d}v\mathrm{d}x.\\
\end{align*} 
This integral is the same than the one appearing in \cite{MR3369905}, where it is proved that 
$$\int_0^1\int_1^2 \min\left\{ \left|Kv-\frac{\tau}{2}\right|^{-3/2}, \left(\frac{tx}{aq}\right)^{-3/2}\right\} \mathrm{d}v\mathrm{d}x \ll \frac{1}{K^{3/2}} \min\left\{1,\frac{10K}{|\tau|}\right\}t^\epsilon.$$
Moreover, we note that 
$$\int_\R g_t(\sigma + i\nu) W_l(\nu) t^{-5/2} \ll t^{-2+\sigma},$$
and thus (keeping in mind that $t^\epsilon<K<t^{1-\epsilon}$),
$$\int_\R M_t(\sigma+i\nu)W_l(\nu) t^{-i\nu}(E+O(t^{-5/2})) \mathrm{d}\nu \ll \frac{t^{\sigma+\epsilon}}{t^{1/2}K^{3/2}} \min\left\{1,\frac{10K}{|\tau|}\right\}.$$
We now treat the main term. Let $\delta'>0$ to be determined later and examine the contribution from $x<1/K^{1-\delta'}$. Using (\ref{extint}) and that $u^\alpha \check{U}(u), v^\alpha V(v) \ll 1,$ for all $\alpha \in \R$, (and thus $t(x-ma)(aq)^{-1}\gg t^{1-\epsilon}$), we estimate
\small
\begin{align*}
&\left(\frac{aq}{t}\right)^{1/2} \int_0^{K^{\delta'-1}}\int_\R V(v) V\left(\frac{(Kv-\frac{\tau}{2})aq}{2\pi tx}\right) \left(\frac{aq}{t(x-ma)}\right)^{1-\sigma} \left|\int_\R g(\nu) e(f(\nu))\mathrm{d}\nu\right| \mathrm{d}v\frac{\mathrm{d}x}{x^{1/2}}\\
%&\ll t^\epsilon \left(\frac{aq}{t}\right)^{1/2} \int_0^{K^{\delta'-1}} \int_{Kv-\frac{\tau}{2} \asymp \frac{tx}{aq}} V(v) \frac{t^\sigma}{t^{1/2}x^{1/2}} \mathrm{d}v\mathrm{d}x\\
&\ll t^\epsilon \int_0^{K^{\delta'-1}} \int_{Kv-\frac{\tau}{2}\asymp \frac{tx}{aq}} V(v) \frac{t^\sigma}{t^{1/2} (Kv-\frac{\tau}{2})^{1/2}} \mathrm{d}v\mathrm{d}x%\\
%&\ll t^{\sigma+\epsilon} \int_0^{K^{\delta'-1}} \frac{x^{1/2}}{(aq)^{1/2}K} \mathrm{d}x \\
 \ll \frac{t^{1/2+\sigma+\epsilon}}{K^{3-\epsilon}aq},
\end{align*}  
\normalsize
upon taking $\delta' = 2\epsilon/3$. We now look at the contribution from $x\in [K^{\delta'-1},1].$ We now reset temporarily 
$$g(v)= (\nu+Kv)^{1/2-\sigma} \left( \frac{aq}{t(x-ma)}\right)^{1-\sigma} V(v)V\left(\frac{(Kv-\frac{\tau}{2})aq}{2\pi tx}\right) \check{U}\left(\frac{(Kv+\nu)aq}{2\pi t(x-ma)}\right),$$
and
$$f(v)= \frac{Kv-\frac{\tau}{2}}{2\pi} \log\left(\frac{(Kv-\frac{\tau}{2})aq}{2e\pi tx}\right) - \frac{Kv+\nu}{2\pi} \log\left(\frac{(Kv+\nu)aq}{2\pi et(x-ma)}\right).$$
 Then, 
 \small
$$f'(v)= -\frac{K}{2\pi} \log\left(\frac{(\nu+Kv)x}{(Kv-\frac{\tau}{2})(x-ma)}\right), f^{(j)}(v)= -\frac{(j-2)!(-K)^j}{2\pi(\nu+Kv)^{j-1}} + \frac{(j-2)!(-K)^j}{2\pi (Kv-\frac{\tau}{2})^{j-1}},$$
\normalsize
for $j\geq 2$, and the stationary point is given by
$$v_0= - \frac{(2\nu+\tau)x -\tau ma}{2Kma}.$$
Now, since $\nu \gg t^{1-\epsilon}$, we have that in the support of the integral,
$$f^{(j)} \asymp \frac{tx}{aq} \left(\frac{Kaq}{tx}\right)^j,$$
for $j \geq 2$, and 
$$g^{(j)}(v) \ll t^{-1/2+\epsilon}\left(1+\frac{Kaq}{tx}\right)^j,$$
for $j\geq 0$. Moreover, we can write
$$f'(v)= \frac{K}{2\pi} \log\left(1+\frac{K(v_0-v)}{\nu+Kv}\right)- \frac{K}{2\pi} \log\left(1+\frac{K(v_0-v)}{Kv-\tau/2}\right),$$
and note that in the support of the integral we have $0\leq Kv-\tau/2 \ll tx/aq \ll K^{1/2}t^{1/2}.$ It follows that if $v_0 \not\in [.5,3]$, then in the support of the integral we have
$$|f'(v)|\gg K \min\left\{1,\frac{Kaq}{tx}\right\}.$$
We now use Lemma 8.1 of \cite{BKY} with 
\begin{align*}
X=t^{-1/2+\epsilon}, & \,U(=V)= \min\left\{1,\frac{tx}{Kaq}\right\},& R=K\min\left\{1,\frac{Kaq}{tx}\right\},\\
&Y=\frac{tx}{aq},& Q=\frac{tx}{Kaq}, 
\end{align*}
so that, choosing $K>t^{1/3+\epsilon}$,
\begin{align*}
\int_\R g(v) e(f(v)) \mathrm{d}v&\ll t^{-1/2+\epsilon} \left[\left(\left(\frac{tx}{aq}\right)^{1/2} \min\left\{1,\frac{Kaq}{tx}\right\}\right)^{-A} +K^{-A}\right]\\
%&\ll t^{-1/2+\epsilon} \left[\left(K\left(\frac{aq}{tx}\right)^{1/2}\right)^{-A} + \left(\frac{tx}{aq}\right)^{-A/2}+K^{-A}\right]\\
&\ll t^{-1/2 +\epsilon} \left[(t^{3\epsilon/4})^{-A} + (K^{\delta'})^{-A} + K^{-A}\right] \ll t^{-B},
\end{align*}
for any $B>0$. In the case where $v_0 \in [.5,3]$, we will use Lemma \ref{BKY}, with $\delta=1/100, A=10000 \delta'^{-1}$ and the same $X,Y,V$ and $Q$ as above. We have
$$\int_\R g(v)e(f(v)) \mathrm{d}v = \frac{e(f(v_0))}{\sqrt{2\pi f''(v_0)}} \sum_{n=0}^{300 A} p_n(v_0) +O_{\delta'}\left(\left(\frac{tx}{aq}\right)^{-A}\right)$$
where $p_n$ is given by (\ref{pn}). Now, since $x\in[K^{\delta'-1},1]$, we have $tx/aq\gg K^{\delta'},$ and therefore the error term is negligible. 

\subsection{Contribution from $n\geq 1$ terms}
We find that
$$f(v_0)= -\frac{\nu+\tau/2}{2\pi} \log\left(-\frac{(\nu+\tau/2)q}{2e\pi tm}\right),$$
and 
\begin{equation}\label{f2v0}
f''(v_0)=\frac{K^2(ma)^2}{2\pi (\nu+\tau/2)(x-ma)x},
\end{equation}
and
\begin{equation}\label{fjv0}
f^{(j)}(v_0) = \frac{(j-2)!(-K)^j(ma)^{j-1}((x-ma)^{j-1}+(-x)^{j-1})}{2\pi (\nu+\tau/2)^{j-1}(ma-x)^{j-1}x^{j-1}}.
\end{equation}
We also find
\begin{align}\label{gv0}
\nonumber g(v_0)&=\left(\frac{tm}{(\nu+\frac{\tau}{2})q}\right)^\sigma \left(\frac{aq}{t}\right) \left(\frac{-(\nu+\frac{\tau}{2})}{(x-ma)ma}\right)^{1/2}\\
&\times V\left(\frac{\tau}{2K}-\frac{(\nu+\frac{\tau}{2})x}{Kma}\right)\check{U}\left(\frac{-(\nu+\frac{\tau}{2})q}{2\pi tm}\right)V\left(-\frac{(\nu+\frac{\tau}{2})q}{m2\pi t}\right).
\end{align}
We wish to keep the term $n=0$ and show that the terms with $n\geq 1$ can be absorbed into an error term. We thus look to bound
\begin{align*}
\int_\R M_t(\sigma+i\nu) W_l(\nu) t^{-i\nu} \frac{e(f(v_0))}{\sqrt{f''(v_0)}} p_n(v_0) \mathrm{d}\nu = \int_\R \tilde{g_n}(\nu) e(\tilde{f}(\nu))\mathrm{d}\nu,
\end{align*}
where
$$\tilde{g_n}(\nu) := \frac{\sqrt{2\pi}(x-ma)^{1/2}x^{1/2}}{Kma} g_t(\sigma+i\nu) W_l(\nu) (\nu+\tau/2)^{1/2} p_n\left(-\frac{(2\nu+\tau)x-\tau ma}{2Kma}\right),$$
and 
$$\tilde{f}(\nu):=  f_t(\sigma +i\nu)- \frac{\nu}{2\pi} \log t- \frac{\nu+\frac{\tau}{2}}{2\pi} \log\left(-\frac{(\nu+\frac{\tau}{2})q}{2e\pi tm}\right).$$
We compute
$$\tilde{f}'(\nu) = f_t'(\sigma+i\nu) - \frac{\log t}{2\pi} - \frac{1}{2\pi} \log\left(-\frac{(\nu+\frac{\tau}{2})q}{2e\pi tm}\right)-\frac{1}{2\pi},$$
and 
$$\tilde{f}''(\nu) = f_t''(\sigma+i\nu) - \frac{1}{2\pi (\nu+\frac{\tau}{2})}.$$
In order to estimate the size of $\tilde{g_n}$, we estimate first
$$p_1(v_0) \ll \frac{g''(v_0)}{f''(v_0)} \ll \frac{XQ^2}{V^2Y},\hspace{1 cm} p_2(v_0)\ll \frac{XQ^4}{V^4Y^2} + \frac{XQ}{VY} + \frac{X}{Y},$$
while, by (\ref{pcontrol}), for $n\geq 3$  we have
$$p_n(v_0) \ll X\left(\left(\frac{V^2Y}{Q^2}\right)^{-n}+Y^{-n/3}\right).$$
We now distinguish two cases. If $x\leq \frac{Kaq}{t}$, then $V=Q= \frac{tx}{Kaq}$, and thus
$$p_n(v_0) \ll \frac{X}{Y},$$
for all $n\geq 1,$ since $Y= \frac{tx}{aq}\gg K^{\delta'}.$ We then show by (\ref{pcontrol}) that
$$\tilde{g_n}'(\nu)\ll \frac{(x-ma)^{1/2}x^{1/2}(\nu+\frac{\tau}{2})^{1/2}}{Kma\nu^{3/2-\sigma}}\frac{X}{Y},$$
so that by the second derivative bound for oscillatory integrals (using that $q\asymp m$, by the support of $\check{U}$),
$$\int_\R \tilde{g_n}(\nu)e(f(\nu))\mathrm{d}\nu \ll \frac{t^{\sigma+\epsilon}(aq)^{1/2} }{Ktx^{1/2}}.$$
Therefore the total contribution from this part is dominated by
$$\left(\frac{aq}{t}\right)^{1/2} \int_{K^{\delta'-1}}^1 \frac{t^{\sigma+\epsilon}(aq)^{1/2}}{Ktx} \mathrm{d}x \ll \frac{t^{\sigma+\epsilon}}{K^2t^{1/2}}.$$
For $x> \frac{Kaq}{t},$ we have $V=1$, and so 
$$p_n(v_0) \ll \frac{t^{1/2+\epsilon}x}{K^2aq}.$$
In this region, we first pass the $x$ integral inside the $\nu$-integral, and since the phase does not depend on $x$, the same analysis holds, replacing $\tilde{g}_n(\nu)$ by
\begin{align*}
\hat{g_n}(\nu):=\left(\frac{aq}{t}\right)^{1/2}\int_{\max\{K^{-1+\delta'},Kaq/t\}}^{1}\frac{1}{\sqrt{x}}  \tilde{g}_n(\nu) \mathrm{d}x.
\end{align*}
We have, using that $m\asymp q$,
$$\hat{g_n}(\nu)\ll  \left(\frac{aq}{t}\right)^{1/2}\tilde{g_n}(\nu) \ll \frac{t^{\sigma+\epsilon}}{K^3aq}.$$
In order to control $\hat{g_n}'(\nu),$ we will first execute the $x$-integral, using integration by parts. Looking at the definition of $p_n$, we note that it is a rational function in $f''(v_0),f'''(v_0),\cdots, g(v_0),g'(v_0)\cdots$ and will describe what the terms of $p_n$ depending on $x$ look like. We first recall that by (\ref{pn}) and (\ref{G}), 
$$p_n(v_0)=\frac{\sqrt{2\pi}e^{\pi i/4}}{n!} \left(\frac{i}{2\tilde{f}^{''}(v_0)}\right)^n G^{(2n)}(v_0) ,$$
where $G^{(2n)}(v_0)$ is a linear combination of elements of the form
$$\hat{g}_n^{(l_0)}(v_0) \tilde{f}^{(l_1)}(v_0)\cdots \tilde{f}^{(l_j)}(v_0),$$
where $l_0+\cdots +l_j= 2n.$ Using (\ref{f2v0}), (\ref{fjv0}) and (\ref{gv0}), we therefore have that those terms of $p_n$ depending on $x$ are of the shape
$$x^i (x-ma)^{j+1/2} V^{(l)}\left(\frac{\tau}{2K}-\frac{(\nu+\frac{\tau}{2})x}{Kma}\right),$$
for some $i,j\geq 1$ and $l\geq 0$. We thus compute
\small
\begin{align*}
&\frac{\mathrm{d}}{\mathrm{d}\nu}\int_{\max\{K^{-1+\delta'},Kaq/t\}}^1 x^{i-1/2} (x-ma)^{j+1/2} V^{(l)}\left(\frac{\tau}{2K}-\frac{(\nu+\tau/2)x}{Kma}\right) \mathrm{d}x\\
&=\frac{\mathrm{d}}{\mathrm{d}\nu} \left(\left[x^{i-1/2}(x-ma)^{j+1/2} \frac{-Kma}{t\nu+\tau/2}V^{(l-1)}\left(\frac{\tau}{2K}-\frac{(\nu+\tau/2)x}{Kma}\right)\right]_{\max\{K^{-1+\delta'},Kaq/t\}}^1\right.\\
&\left. + \int_{\max\{K^{-1+\delta'},Kaq/t\}}^1 (x^{i-1/2}(x-ma)^{j+1/2})' \frac{Kma}{t\nu+\tau/2}V^{(l-1)}\left(\frac{\tau}{2K}-\frac{(\nu+\tau/2)x}{Kma}\right)\right)\mathrm{d}x\\
&\ll \frac{(ma)^{j+1/2}}{\nu+\frac{\tau}{2}}.
\end{align*}
\normalsize
These calculations show that 
$$\int_\R \left|\frac{\mathrm{d}}{\mathrm{d}\nu}\hat{g_n}(\nu)\right| \mathrm{d}\nu \ll t^\epsilon \hat{g_n} \ll\frac{t^{\sigma+\epsilon}}{K^3aq}, $$
and by the second derivative bound for oscillatory integrals, 
$$\int_\R \hat{g_n}(\nu) e(\tilde{f}(\nu)) \mathrm{d}\nu \ll \frac{t^{1/2+\sigma + \epsilon}}{K^3aq},$$
which is the same bound we obtained for $x\in (0, K^{\delta'-1}).$ We therefore obtain
\begin{align*}
\left(\frac{aq}{t}\right)^{1/2}\int_0^1\int_\R g(v)e(f(v))\mathrm{d}v\frac{\mathrm{d}x}{x^{1/2}} &= \left(\frac{aq}{t}\right)^{1/2} \int_{K^{-1+\delta'}}^1 \frac{g(v_0)e(f(v_0)+1/8)}{x^{1/2}\sqrt{f''(v_0)}}\mathrm{d}x + E^{*},
\end{align*}
where $E^*$ is an error term such that
\begin{align*}
\int_\R M_t(\sigma+i\nu) W_l(\nu) t^{-i\nu} E^{*} \mathrm{d}\nu \ll \frac{t^{1/2+\sigma+\epsilon}}{aqK^3}.
\end{align*}
Now, plugging in the value for $v_0$, we get that the leading term above reduces to 
\begin{align*}
&c_3\frac{\nu+\frac{\tau}{2}}{K} \left(\frac{-q}{mt}\right)^{3/2}V\left(\frac{-(\nu+\frac{\tau}{2})q}{2\pi mt}\right) \left(-\frac{(\nu+\frac{\tau}{2})q}{2e\pi tm}\right)^{-i(\nu+\tau/2)} \left(\frac{tm}{(\nu+\frac{\tau}{2})q}\right)^\sigma \\
&\times \check{U}\left(-\frac{(\nu+\frac{\tau}{2})q}{2\pi tm}\right)  \int_{K^{-1+\delta'}}^1 V\left(\frac{\tau}{2K}-\frac{(\nu+\frac{\tau}{2})x}{Kma}\right)\mathrm{d}x,
\end{align*}
for some absolute constant $c_3$. Set 
$$B(C,\tau,\nu)= t^{-5/2}+ E+E^{*},$$
and note that
\begin{equation}\label{bctau}
\int_{-\frac{(tK)^{1/2}t^\epsilon}{C}}^{\frac{(TK)^{1/2}t^\epsilon}{C}} \int_\R M_t(\sigma+i\nu)W_l(\nu)t^{-i\nu}B(C,\tau,\nu)\mathrm{d}\nu\mathrm{d}\tau \ll \frac{t^{\sigma+\epsilon}}{t^{1/2}K^{1/2}} \left(1+\frac{t}{C^2 K^{3/2}}\right).
\end{equation}
We may now derive from these computations the following: 
\begin{lemma}
We have 
$$I^{**}(q,m,\tau,\nu)= I_1(q,m,\tau,\nu) + I_2(q,m,\tau,\nu),$$ 
where
\small
\begin{align*}
I_1(q,m,\tau,\nu) &= \frac{c_4}{(\nu+\frac{\tau}{2})^{1/2}K} \left(-\frac{(\nu+\frac{\tau}{2})q}{2\pi etm}\right)^{3/2-i(\nu+\frac{\tau}{2})} V\left(-\frac{(\nu+\frac{\tau}{2})q}{2\pi mt}\right) \left(\frac{tm}{(\nu+\frac{\tau}{2})aq}\right)^\sigma\\
& \times \check{U}\left(-\frac{(\nu+\frac{\tau}{2})q}{2\pi mt}\right) \int_{K^{-1+\delta'}}^1 V\left(\frac{\tau}{2K}- \frac{(\nu+\frac{\tau}{2})x}{Kma}\right)\mathrm{d}x,
\end{align*}
\normalsize
where $c_4$ is an absolute constant, and 
$$I_2(q,m,\tau,\nu):= I^{**}(q,m,\tau,\nu) - I_1(q,m,\tau,\nu) = B(C,\tau,\nu).$$
\end{lemma}
Consequently from Lemma \ref{lemmastc} we arrive at:
\begin{lemma}
We have 
$$S_l(t,C)= \sum_{J\in \mathcal{J}}\left\{S_{1,J}(t,C)+S_{2,J}(t,C)\right\}+O(t^{-1000}),$$
where
\begin{align*}
S_{r,J}(t,C)&=\frac{iKt^{1/2}}{4\pi^2} \sum_\pm \sum_{n\ll Kt^\epsilon} \frac{\r_f(\mp n)}{\sqrt{n}}\\
& \sum_{C<q\leq 2C}\sum_{\substack{(m,q)=1 \\ 1\leq |m|\ll qt^\epsilon}} \frac{e\left(\pm \frac{n\bar{m}}{q}\right)}{aq}I_{r,J,\pm}(q,m,n),
\end{align*}
and
\small
\begin{align*}
I_{r,J,\pm}(q,m,n)= \int_{\R^2}  M_t(\sigma+i\nu)W_l(\nu)t^{-i\nu} \left(\frac{2\pi\sqrt{nt}}{q}\right)^{-i\tau} \gamma_\pm(i\tau+1) I_r(q,m,\tau,\nu)W_J(\tau)\mathrm{d}\tau\mathrm{d}\nu.
\end{align*}
\end{lemma}

\normalsize
\subsection{Application of Cauchy and Poisson I}
We will estimate here 
$$\tilde{S}_2(t,C) := \sum_{J \in \mathcal{J}} S_{2,J}(t,C).$$
Taking a dyadic subdivision and using the bound $|\gamma_\pm(i\tau +1)|\ll 1$, we get
\begin{align*}
\tilde{S}_2(t,C)&\ll K t^{1/2} \int_{\frac{-(tK)^{1/2}t^\epsilon}{C}}^{\frac{(tK)^{1/2}t^\epsilon}{C}} \sum_\pm\sum_{\substack{1\leq L \ll Kt^\epsilon \\ L \hbox{ dyadic}}} \sum_{n\in \Z} \frac{|\r_f(\mp n)|}{\sqrt{n}} U\left(\frac{n}{L}\right)\\
& \times \left|\sum_{C<q\leq 2C} \sum_{\substack{(m,q)=1 \\ 1\leq |m| \ll qt^\epsilon}} \frac{e\left(\pm\frac{n\bar{m}}{q}\right)}{aq^{1-i\tau}} B(C,\tau)\right|\mathrm{d}\tau,
\end{align*}
where
\begin{align*}
B(C,\tau) &:= \int_\R M_t(\sigma+i\nu)W_l(\nu) t^{-i\nu}B(C,\tau,\nu) \mathrm{d}\nu\\
&\ll t^{\sigma+\epsilon} \left(\frac{1}{t^{1/2}K^{3/2}} \min\left\{1,\frac{10K}{|\tau|}\right\}+\frac{1}{CK^{5/2}}\right)
\end{align*}
By Cauchy and Rankin-Selberg, we get
\begin{align*}
\tilde{S}_2(t,C)&\ll Kt^{1/2+\epsilon}\int_{-\frac{(tK)^{1/2}t^\epsilon}{C}}^{\frac{(tK)^{1/2}t^\epsilon}{C}} \sum_\pm \sum_{\substack{1\leq L\ll Kt^\epsilon\\ L \hbox{ dyadic}}} L^{1/2} \left[S_{2,\pm}(t,C,L,\tau)\right]^{1/2}\mathrm{d}\tau,
\end{align*}
where
\begin{align*}
S_{2,\pm}(t,C,L,\tau)&= \sum_{n\in \Z} \frac{1}{n} U\left(\frac{n}{L}\right) \left|\sum_{C<q\leq 2C} \sum_{\substack{(m,q)=1 \\ 1\leq |m| \ll qt^\epsilon}} \frac{e\left(\pm \frac{n\bar{m}}{q}\right)}{aq^{1-i\tau}} B(C,\tau)\right|^2.
\end{align*}
Opening the absolute square and interchanging the order of summation, we obtain
$$S_{2,\pm}(t,C,L,\tau) = \sum_{C<q,q'\leq 2C} \sum_{\substack{(m,q)=1 \\ 1\leq |m|\ll qt^\epsilon}} \sum_{\substack{(m',q')=1\\ 1\leq |m'|\ll q't^\epsilon}} \frac{|B(C,\tau)|^2}{aa'q^{1-i\tau}q'^{1+i\tau}} T,$$
where
$$T:= \sum_{n\in \Z} \frac{1}{n} U\left(\frac{n}{L}\right) e\left(\pm\frac{n\bar{m}}{q} \mp \frac{n\bar{m'}}{q'}\right).$$
Splitting in congruence classes mod $qq'$ and applying Poisson summation, we get
\begin{align*}
T%&= \sum_{\beta \mod qq'} e\left(\pm \frac{\beta\bar{m}}{q}\mp \frac{\beta\bar{m'}}{q'}\right) \sum_{n\in \Z} \frac{1}{\beta + nqq'} U\left(\frac{\beta+nqq'}{L}\right)\\
%&= \sum_{\beta \mod qq'} e\left(\pm \frac{\beta \bar{m}}{q}\mp \frac{\beta\bar{m'}}{q'}\right) \sum_{n\in\Z} \int_\R \frac{1}{\beta+yqq'}U\left(\frac{\beta+yqq'}{L}\right) e(-ny)\mathrm{d}y\\
%&= \sum_{n\in \Z} \sum_{\beta \mod qq'} e\left(\pm \frac{\beta\bar{m}}{q}\mp \frac{\beta\bar{m'}}{q'}+\frac{\beta n}{qq'}\right) \frac{1}{qq'} \int_\R \frac{1}{y} U(y) e\left(-\frac{Lny}{qq'}\right)\mathrm{d}y\\
&= \sum_{n\in \Z} \delta_{\pm q'\bar{m}\mp q\bar{m'}+n \equiv 0 \mods{qq'}} \int_\R\frac{1}{y}U(y)e\left(-\frac{Lny}{qq'}\right).
\end{align*}
We may now truncate the $n$-sum to $n\ll C^2t^\epsilon/L$, for otherwise the oscillatory integral is negligibly small. We may therefore estimate
\begin{align*}
S_{2,\pm}(t,C,L,\tau)& \ll  \sum_{C<q,q'\leq 2C}\sum_{\substack{(m,q)=1\\ 1\leq |m|\ll qt^\epsilon}} \sum_{\substack{(m',q')=1\\ 1\leq |m'|\ll q't^\epsilon}} \sum_{\substack{n\ll \frac{C^2t^\epsilon}{L}\\ n\equiv \pm q\bar{m'} \mp q'\bar{m} \mods{qq'}}} \frac{K|B(C,\tau)|^2}{tC^2}\\
& \ll \frac{t^\epsilon C^3K |B(C,\tau)|^2}{tL}.
\end{align*}
Thus, by (\ref{bctau}), we have
\begin{align*}
\tilde{S}_2(t,C) & \ll  \sum_{\substack{1\leq L \ll Kt^\epsilon\\ L \hbox{ dyadic}}} C^{3/2}K^{3/2} t^\epsilon \int_{-\frac{(tK)^{1/2}t^\epsilon}{C}}^{\frac{(tK)^{1/2}t^\epsilon}{C}}|B(C,\tau)|\mathrm{d}\tau\\
&\ll t^{\sigma+\epsilon} \left(\frac{C^{3/2}K}{t^{1/2}}+\frac{t^{1/2}}{C^{1/2}K^{1/2}}\right).
\end{align*}
The contribution of $S_2(t,C)$ to $S_l^+(t)$ is therefore bounded by
$$t^\epsilon \left(\frac{t^{5/4}}{K^{3/4}}+ \frac{t^{3/2}}{K^{3/2}}\right).$$
Upon taking $K=t^{1/2},$ we note that this is bounded by $t^{1-1/8+\epsilon}$.

\subsection{Application of Poisson and Cauchy II}
The analysis for $S_{1,J}$ is more delicate as we need to exploit some cancelation coming from both the $\nu$ and $\tau$ integrals. The idea is to use Cauchy and Rankin-Selberg as before, but keeping the integrals over $\tau$ and $\nu$ inside. We may bound
\begin{align*}
S_{1,J}(t,C) \ll Kt^{1/2} \sum_\pm \sum_{\substack{1\leq L \ll Kt^\epsilon \\ L \hbox{ dyadic }}} L^{1/2} \left[S_{1,J,\pm}(t,C,L)\right]^{1/2},
\end{align*}
where  
\small
\begin{align*}
S_{1,J,\pm}(t,C,L) =& \sum_{n\in \Z} \frac{1}{n} U\left(\frac{n}{L}\right) \left| \int_\R \int_\R  M_t(\sigma+i\nu) W_l(\nu) t^{-i\nu} (2\pi \sqrt{nt})^{-i\tau} \right.\\
& \left. \times \gamma_\pm(i\tau+1) \sum_{C<q\leq 2C} \sum_{\substack{(m,q)=1 \\ 1\leq |m| \ll qt^\epsilon}} \frac{e\left(\pm \frac{n\bar{m}}{q}\right)}{aq^{1-i\tau}} I_1(q,m,\tau,\nu) W_J(\tau) \mathrm{d}\tau\mathrm{d}\nu \right|^2.
\end{align*}
\normalsize
Opening the absolute square and interchanging the order of summation, we find that $S_{1,J,\pm}(t,C,L)$ is given by
\small
\begin{align*}
& \int_{\R^4} M_t(\sigma +i\nu) \overline{M_t(\sigma+i\nu')} W_l(\nu)W_l(\nu') t^{i(\frac{\tau'-\tau}{2}+\nu'-\nu)} \gamma_{\pm}(1+i\tau) \overline{\gamma_\pm(1+i\tau')}W_J(\tau)W_J(\tau')\\
&\sum_{C<q,q'\leq 2C} \sum_{\substack{(m,q)=1\\ 1\leq |m|\ll qt^\epsilon}} \sum_{\substack{(m',q')=1\\ 1\leq |m'|\ll q't^\epsilon}} \frac{I_1(q,m,\tau,\nu)\overline{I_1(q',m',\tau',\nu')}}{aa'(2\pi)^{i(\tau-\tau')}q^{1-i\tau}q'^{1+i\tau'}} \mathcal{T}' \mathrm{d}\tau\mathrm{d}\tau' \mathrm{d}\nu \mathrm{d}\nu',
\end{align*}
\normalsize
where
$$\mathcal{T}' = \sum_{n\in \Z} \frac{1}{n^{1+i\frac{\tau-\tau'}{2}}} U\left(\frac{n}{L}\right) e\left(\pm \frac{n\bar{m}}{q} \mp \frac{n\bar{m'}}{q'}\right).$$
Applying Poisson summation, similarly to the previous section, we obtain
$$\mathcal{T}' = \frac{L^{i\frac{\tau'-\tau}{2}}}{qq'} \sum_{n\in \Z} \delta\pm(n,m,m',q,q') U^\dagger\left(\frac{nL}{qq'}, -i\frac{\tau-\tau'}{2}\right),$$
where
$$\delta_\pm(n,m,m',q,q')= qq' \delta_{\pm q'\bar{m} \mp q\bar{m'} + n \equiv 0 \mods{qq'}}.$$
Since $|\tau-\tau'|\ll (tK)^{1/2}t^\epsilon/C$ and $q,q'\asymp C$, we have by Lemma \ref{Munshilemma} that if $|n|\gg C(tK)^{1/2} t^\epsilon/L,$ then the contribution is negligibly small.

\begin{lemma}
The sum $S_{1,J,\pm}(t,C,L)$ is dominated by the sum
$$\frac{K}{tC^2} \sum_{C<q,q' \leq 2C} \sum_{\substack{(m,q)=1 \\ 1\leq |m| \ll qt^\epsilon}} \sum_{\substack{(m',q')=1\\ 1\leq |m'| \ll q't^\epsilon}} \sum_{\substack{|n|\ll C(tK)^{1/2}t^\epsilon/L\\ n\equiv \pm q\bar{m'}\mp q'\bar{m} \mods{qq'}}} |\mathcal{K}_\pm| + O(t^{-1000}),$$
where
\begin{align*}
\mathcal{K}_\pm&= \int_{\R^4} M_t(\sigma+i\nu) \overline{M_t(\sigma+i\nu')} W_l(\nu)W_l(\nu') t^{i(\nu'-\nu)} \frac{(4\pi^2 tL)^{-i\frac{\tau-\tau'}{2}}}{q^{-i\tau}q'^{i\tau'}} W_J(\tau)W_J(\tau')\\
& \gamma_\pm(i\tau+1) \overline{\gamma_\pm(i\tau'+1)} I_1(q,m,\tau,\nu) \overline{I_1(q',m',\tau',\nu')} U^\dagger\left(\frac{nL}{qq'},i\frac{\tau'-\tau}{2}\right) \mathrm{d}\tau\mathrm{d}\tau'\mathrm{d}\nu\mathrm{d}\nu'.
\end{align*}
\end{lemma}
We are thus only left with understanding $\mathcal{K}_\pm$. Writing out explicitly $I_1(q,m,\tau,\nu)$, we obtain
\begin{align*}
\mathcal{K}_\pm &= \frac{|c_4|^2}{K^2}  \int_{\R^4} W_J(q,m,\tau,\nu)\overline{W_J(q',m',\tau',\nu')} e(f_t(\sigma+i\nu)-f_t(\sigma+i\nu'))\\
&\times t^{i(\nu'-\nu)}U^\dagger\left(\frac{nL}{qq'},i\frac{\tau'-\tau}{2}\right) \left(-\frac{(\nu+\frac{\tau}{2})q}{2\pi etm}\right)^{-i(\nu+\frac{\tau}{2})} \left(-\frac{(\nu'+\frac{\tau'}{2})q'}{2\pi e tm'}\right)^{i(\nu'+\frac{\tau'}{2})}\\
&\times \gamma_\pm(1+i\tau)\overline{\gamma_\pm(1+i\tau')} \frac{(4\pi^2tL)^{i\frac{\tau'-\tau}{2}}}{q^{-i\tau}q'^{i\tau'}} \mathrm{d}\tau\mathrm{d}\tau'\mathrm{d}\nu\mathrm{d}\nu',
\end{align*}
where
\begin{align*}
W_J(q,m,\tau,\nu)&= g_t(\sigma +i\nu) \frac{W_l(\nu) W_J(\tau)}{(\nu+\frac{\tau}{2})^{1/2}}  \left(-\frac{(\nu+\frac{\tau}{2})q}{2\pi etm}\right)^{3/2} V\left(-\frac{(\nu+\frac{\tau}{2})q}{2\pi tm}\right) \\
&\times\left(\frac{tm}{(\nu+\frac{\tau}{2})q}\right)^\sigma \check{U}\left(-\frac{(\nu+\frac{\tau}{2})q}{2\pi mt}\right) \int_{K^{\sigma'-1}}^1 V\left(\frac{\tau}{2K}-\frac{(\nu+\frac{\tau}{2})x}{Kma}\right)\mathrm{d}x.
\end{align*}
We note in passing the following estimates 
\begin{equation}\label{w1}
\frac{\mathrm{d}}{\mathrm{d}\tau} W_J(q,m,\tau,\nu) \ll \frac{|\nu|^{\sigma-1}}{|\tau|},
\end{equation}
and
\begin{equation}\label{w2}
\frac{\mathrm{d}}{\mathrm{d}\nu}W_J(q,m,\tau,\nu) \ll |\nu|^{\sigma-2}.
\end{equation}
We first analyse the case $n=0$; it will be sufficient to consider 
\begin{align*}
&\int_{\R^2} W_J(q,m,\tau,\nu)\overline{W_J(q',m',\tau',\nu')} e(f_t(\sigma+i\nu)-f_t(\sigma+i\nu')) t^{i(\nu'-\nu)}\\
&\times \left(-\frac{(\nu+\frac{\tau}{2})q}{2\pi etm}\right)^{-i(\nu+\frac{\tau}{2})} \left(-\frac{(\nu'+\frac{\tau'}{2})q'}{2\pi etm'}\right)^{i(\nu'+\frac{\tau'}{2})}\mathrm{d}\nu\mathrm{d}\nu'\\
&= \int_{\R^2} W_J(q,m,\tau,\nu)\overline{W_J(q',m',\tau',\nu')} e(f(\nu,\nu'))\mathrm{d}\nu\mathrm{d}\nu',
\end{align*}
where we temporarily define
\begin{align*}
f(\nu,\nu')&= f_t(\sigma+i\nu)-f_t(\sigma+i\nu') +\frac{\nu'-\nu}{2\pi} \log t \\
&-\frac{\nu+\frac{\tau}{2}}{2\pi} \log\left(-\frac{(\nu+\frac{\tau}{2})q}{2\pi etm}\right) +\frac{\nu'+\frac{\tau'}{2}}{2\pi} \log\left(-\frac{(\nu'+\frac{\tau'}{2})q'}{2\pi e tm'}\right).
\end{align*}
We compute
$$\frac{\mathrm{d}f}{\mathrm{d}\nu} = f_t'(\sigma+i\nu) - \frac{\log t}{2\pi} - \frac{1}{2\pi} \log\left(-\frac{(\nu+\frac{\tau}{2})q}{2\pi etm}\right) - \frac{1}{2\pi},$$
$$\frac{\mathrm{d}f}{\mathrm{d}\nu'} = -f_t'(\sigma+i\nu') + \frac{\log t}{2\pi} + \frac{1}{2\pi} \log\left(-\frac{(\nu'+\frac{\tau'}{2})q'}{2\pi etm'}\right) +\frac{1}{2\pi}.$$
and thus
$$\frac{\mathrm{d}^2f}{\mathrm{d}\nu\mathrm{d}\nu'} = 0,$$
while by (\ref{condad}), we have
$$\frac{\mathrm{d}^2f}{\mathrm{d}\nu^2} =f_t''(\sigma+i\nu) - \frac{1}{2\pi (\nu+\frac{\tau}{2})} \gg |\nu|^{-1},$$
and
$$\frac{\mathrm{d}^2f}{\mathrm{d}\nu'^2} = -f_t''(\sigma+i\nu') + \frac{1}{2\pi (\nu'+\frac{\tau'}{2})} \gg |\nu'|^{-1}.$$
We also note that by (\ref{w2}), we have
$$\hbox{Var}(W_J(q,m,\tau,\nu)\overline{W_J(q',m',\tau',\nu')})\ll t^{2\sigma-2+\epsilon}.$$
We now have by the second derivative bound for oscillatory integrals in multivariables (see \cite{Srinstat}) that 
\begin{equation}\label{2var}
\int_{\R^2} W_J(q,m,\tau,\nu)\overline{W_J(q',m',\tau',\nu')} e(f(\nu,\nu'))\mathrm{d}\nu\mathrm{d}\nu' \ll t^{2\sigma-1+\epsilon}.
\end{equation}
By integration by parts, if $|\tau-\tau'| \gg t^\epsilon$, then $U^\dagger\left(0,i\frac{\tau-\tau'}{2}\right)$ is negligibly small. The contribution from $n=0$ to $\mathcal{K_\pm}$ is thus bounded by
$$K^{-2} \int\int_{\substack{|\tau-\tau'|\ll t^\epsilon\\ |\tau|,|\tau'|\asymp J}} t^{2\sigma-1+\epsilon} \ll  \frac{t^{2\sigma+\epsilon}}{Ct^{1/2}K^{3/2}}.$$ 
We now treat the case $n\not = 0$. We have by Lemma 5 of \cite{MR3369905} that
\begin{align*}
U^\dagger\left(\frac{nL}{qq'}, -i\frac{\tau-\tau'}{2}\right) & = \frac{c_5}{(\tau'-\tau)^{1/2}} U\left(\frac{(\tau'-\tau)qq'}{4\pi nL}\right) \left(\frac{(\tau'-\tau)qq'}{4\pi enL}\right)^{-i(\tau-\tau')/2}\\
& +O\left(\min\left\{\frac{1}{|\tau-\tau'|^{3/2}}, \frac{C^3}{(|n|L)^{3/2}}\right\}\right),
\end{align*}
for some constant $c_5$ (which depends on the sign of $n$). In order to bound the error term, we use (\ref{2var}) to see that the contribution is bounded by
$$\frac{t^{2\sigma-1+\epsilon}}{K^2} \int_{[J,2J]^2} \min\left\{\frac{1}{|\tau-\tau'|^{3/2}}, \frac{C^3}{(|n|L)^{3/2}}\right\}.$$
We first estimate
\begin{align*}
\frac{t^{2\sigma-1+\epsilon}}{K^2} \int_{\substack{[J,2J]^2 \\ |\tau-\tau'|\leq  |nL|/C^2}} \frac{C^3}{(|n|L)^{3/2}} \mathrm{d}\tau\mathrm{d}\tau' \ll \frac{t^{2\sigma-1+\epsilon}CJ}{K^2(|n|L)^{1/2}}\ll \frac{t^{2\sigma-1/2+\epsilon}}{K^{3/2}(|n|L)^{1/2}},
\end{align*}
and then
\begin{align*}
 \frac{t^{2\sigma-1+\epsilon}}{K^2} \int_{\substack{[J,2J]^2\\ |\tau-\tau'|> |nL|/C^2}} \frac{1}{|\tau-\tau'|^{3/2}} \mathrm{d}\tau\mathrm{d}\tau' &\ll \frac{Ct^{2\sigma-1+\epsilon}}{K^2(|nL|)^{1/2}} \int_{[J,2J]^2} \frac{1}{|\tau-\tau'|^{1-\epsilon}}\mathrm{d}\tau\mathrm{d}\tau'\\
& \ll \frac{CJt^{2\sigma-1+\epsilon}}{K^2(|nL|)^{1/2}} \ll \frac{t^{2\sigma-1/2+\epsilon}}{K^{3/2}(|nL|)^{1/2}}.
\end{align*}
We thus set 
$$B^*(C,0)= \frac{t^{2\sigma+\epsilon}}{K^{3/2}Ct^{1/2}},$$
and for $n\not= 0$, 
$$B^*(C,n)= \frac{t^{2\sigma+\epsilon}}{K^{3/2}t^{1/2}(|n|L)^{1/2}}.$$

We now consider the main term. As noted in Section 4.1, the contribution from $\gamma_+$ is simpler, and thus we will only focus on $\gamma_-$. We first note that by Fourier inversion, we have
$$\left(\frac{4\pi nL}{(\tau'-\tau)qq'}\right)^{1/2}  U\left(\frac{(\tau'-\tau)qq'}{4\pi nL}\right) = \int_\R U^\dagger\left(r,\frac{1}{2}\right) e\left(r\frac{(\tau'-\tau)qq'}{4\pi nL}\right)\mathrm{d}r.$$
Pulling out the oscillation from the $\gamma_-$ factors, we conclude that for some constant $c_6$ (depending on the sign of $n$), we have
\begin{align*}
\mathcal{K}_- =& \frac{c_6}{K^2} \left(\frac{qq'}{|n|L}\right)^{1/2} \int_\R U^\dagger\left(r,\frac{1}{2}\right) \int_{\R^4} g(\tau,\tau',\nu,\nu') e(f(\tau,\tau',\nu,\nu',r)) \mathrm{d}\tau\mathrm{d}\tau' \mathrm{d}\nu\mathrm{d}\nu'\mathrm{d}r\\
&+ O(B^*(C,n)),
\end{align*}
where
\small
\begin{align*}
f(\tau,\tau',\nu,\nu',r) =& f_t(\sigma+i\nu) - f_t(\sigma+i\nu') + \frac{\nu'-\nu}{2\pi}\log t +\frac{\tau}{2\pi}\log\left(\frac{|\tau|}{e}\right) - \frac{\tau'}{2\pi} \log\left(\frac{|\tau'|}{e}\right)\\
& + \frac{\tau'-\tau}{4\pi} \log\left(\frac{(\tau'-\tau)4\pi t qq'}{en}\right) - \frac{\nu+\frac{\tau}{2}}{2\pi} \log\left(-\frac{(\nu+\frac{\tau}{2})q}{2\pi etm}\right) + \frac{\tau}{2\pi} \log q\\
&- \frac{\tau'}{2\pi} \log q' + \frac{\nu'+\frac{\tau'}{2}}{2\pi} \log\left(-\frac{(\nu'+\frac{\tau'}{2})q'}{2\pi etm'}\right) + \frac{r(\tau'-\tau)qq'}{4\pi nL},
\end{align*}
\normalsize
and
$$g(\tau,\tau',\nu,\nu') = W_J(q,m,\tau,\nu)\overline{W_J(q',m',\tau',\nu')} \Phi_-(\tau) \overline{\Phi_-(\tau')}.$$
We will use the second derivative bound for multivariable oscillatory integrals as can be found in \cite{Srinstat} and hence compute
\begin{align*}
&2\pi \frac{\mathrm{d}^2f}{\mathrm{d}\tau^2} = \frac{1}{\tau}- \frac{1}{4(\nu+\frac{\tau}{2})} + \frac{1}{2 (\tau'-\tau)}, \, 2\pi \frac{\mathrm{d}^2f}{\mathrm{d}\tau\mathrm{d}\tau'} = \frac{1}{2(\tau-\tau')}, \,2\pi \frac{\mathrm{d}^2f}{\mathrm{d}\tau\mathrm{d}\nu}  = -\frac{1}{2(\nu+\frac{\tau}{2})},\\
& 2\pi \frac{\mathrm{d}^2f}{\mathrm{d}\tau'^2} = -\frac{1}{\tau'} + \frac{1}{4(\nu'+\frac{\tau'}{2})}+\frac{1}{2(\tau'-\tau)}, \, 2\pi \frac{\mathrm{d}^2f}{\mathrm{d}\tau'\mathrm{d}\nu'} = \frac{1}{2(\nu'+\frac{\tau'}{2})},\\
&\frac{\mathrm{d}^2f}{\mathrm{d}\nu^2} = f_t''(\sigma+i\nu) - \frac{1}{2\pi (\nu+\frac{\tau}{2})}, \, \frac{\mathrm{d}^2f}{\mathrm{d}\nu'^2} = \frac{1}{2\pi (\nu'+\frac{\tau'}{2})} - f_t''(\sigma+i\nu'),
\end{align*}
while
$$\frac{\mathrm{d}^2f}{\mathrm{d}\tau\mathrm{d}\nu'}=\frac{\mathrm{d}^2f}{\mathrm{d}\tau'\mathrm{d}\nu}= \frac{\mathrm{d}^2f}{\mathrm{d}\nu\mathrm{d}\nu'} = 0.$$
Computing the minors of the Hessian matrix , we see from \cite[Lemma 5]{Srinstat} that for $D$ a box in $R^4$,
\begin{equation}\label{geq1} 
\int_{D} e(f(\tau,\tau',\nu,\nu')) \mathrm{d}\tau\mathrm{d}\tau'\mathrm{d}\nu\mathrm{d}\nu' \ll t^\epsilon Jt,
\end{equation}
where we used $r_1=r_2=J^{-1/2}$ and $r_3=r_4= t^{-1/2}$ as can be seen from our calculations of the second derivatives and that $\tau,\tau' \in [J,\frac{4}{3} J]$. Using (\ref{w1}) and (\ref{w2}), we compute the total variation, using that $t^{1-\epsilon} \ll|\nu|\ll t^{1+\epsilon}$:
\begin{align}\label{totvar}
\nonumber \hbox{Var}(g(\tau,\tau',\nu,\nu')) & := \int_{\R^4}  \left|\frac{\mathrm{d}g}{\mathrm{d}\tau\mathrm{d}\tau'\mathrm{d}\nu\mathrm{d}\nu'}\right| \mathrm{d}\tau\mathrm{d}\tau'\mathrm{d}\nu\mathrm{d}\nu'\\
\nonumber &\ll \int_{\R^4} \frac{|\nu|^{\sigma-2}|\nu'|^{\sigma-2}}{|\tau||\tau'|} Jt^{1+\epsilon} \mathrm{d}\tau\mathrm{d}\tau'\mathrm{d}\nu\mathrm{d}\nu'\\
&\ll t^{2\sigma-2+\epsilon}.
\end{align} 
 By integration by parts, we note that by (\ref{geq1}) and (\ref{totvar}), we have
\begin{align*}
&\int_{\R^4} g(\tau,\tau',\nu,\nu') e(f(\tau,\tau',\nu,\nu',r)) \mathrm{d}\tau\mathrm{d}\tau' \mathrm{d}\nu\mathrm{d}\nu'\\
&\ll \int_{\R^4} \left|\frac{\mathrm{d}g}{\mathrm{d}\tau\mathrm{d}\tau'\mathrm{d}\nu\mathrm{d}\nu'}\right| Jt^{1+\epsilon}\mathrm{d}\tau\mathrm{d}\tau'\mathrm{d}\nu\mathrm{d}\nu'\\
&\ll J t^{2\sigma-1+\epsilon}.
\end{align*}
Then, integrating trivially over $r$ and using the rapid decay of Fourier transforms, we arrive at the following result:

\begin{lemma}
We have 
$$\mathcal{K}_- \ll  B^*(C,n).$$
\end{lemma}
We now write 
$$S_{1,J,-}(t,C,L) = S_{1,J,-}^\flat(t,C,L) + S_{1,J,-}^\sharp(t,C,L),$$
where $S_{1,J,-}^\flat(t,C,L)$ corresponds to $n=0$ contribution, while $S_{1,J,-}^\sharp(t,C,L)$ corresponds to the $n\not= 0$ frequencies. We first estimate
\begin{align*}
S_{1,J,-}^\flat(t,C,L) &\ll \frac{K}{tC^2} \sum_{C<q,q' \leq 2C} \sum_{\substack{(m,q)=1 \\ 1\leq |m| \ll qt^\epsilon}} \sum_{\substack{(m',q')=1\\ 1\leq |m'| \ll q't^\epsilon}} \delta_{- q \overline{m'} + q' \overline{m} \equiv 0 \mods{qq'}} \frac{t^{2\sigma +\epsilon}}{K^{3/2}Ct^{1/2}}\\
&\ll \frac{t^{2\sigma+\epsilon}}{t^{3/2}K^{1/2}}.
\end{align*}
Taking a dyadic subdivision, we estimate
\small
\begin{align*}
S_{1,J,-}^\sharp(t,C,L)& \ll \frac{K}{tC^2} \sum_{C<q,q'\leq 2C} \sum_{\substack{(m,q)=1\\ 1\leq |m|\ll qt^\epsilon}} \sum_{\substack{(m',q')=1\\ 1\leq |m'|\ll q't^\epsilon}} \sum_{\substack{1\leq |n|\ll \frac{C(tK)^{\frac{1}{2}} t^\epsilon}{L}\\ n\equiv - q\overline{m'}+ q'\overline{m} \mods{qq'}}}B^*(C,n)\\
&\ll\frac{K}{tC^2} \sum_{C<q,q'\leq 2C} \sum_{\substack{(m,q)=1\\ 1\leq |m| \ll qt^\epsilon}} \sum_{\substack{(m',q')=1\\ 1\leq |m'|\ll q't^\epsilon}}  \sum_{\substack{1\leq |n|\ll \frac{C(tK)^{\frac{1}{2}} t^\epsilon}{L}\\ n\equiv q'\overline{m}- q\overline{m'} \mods{qq'}}} \frac{t^{2\sigma +\epsilon}}{K^{3/2}t^{1/2}(|n|L)^{\frac{1}{2}}}\\
&\ll \frac{t^{2\sigma+\epsilon}}{t^{3/2}K^{1/2}C^2L^{1/2}} \sum_{\substack{H\leq \frac{C(tK)^{1/2}t^\epsilon}{L}\\ H \hbox{ Dyadic}}} \sum_{C<q,q'\leq 2C} \sum_{\substack{(m,q)=1 \\ 1\leq |m| \ll qt^\epsilon}}  \\
&\times  \sum_{\substack{(m',q')=1\\ 1\leq|m'|\ll q't^\epsilon}} \sum_{\substack{H< |n| \leq 2H \\ n \equiv -q\overline{m'}+q'\overline{m} \mods{qq'}}} H^{-1/2} \\
& \ll \frac{t^{2\sigma+\epsilon}}{t^{3/2}K^{1/2}C^2L^{1/2}} \sum_{\substack{H\leq \frac{C(tK)^{1/2}t^\epsilon}{L}\\ H \hbox{ Dyadic}}} H^{-1/2} \sum_{C< q,q'\leq 2C} \sum_{H<|n|\leq 2H}\\
&\times \sum_{\substack{1\leq|m| \ll qt^\epsilon\\ (m,q)=1}} \sum_{\substack{1\leq |m'| \ll q't^\epsilon\\ (m',q')=1}}  \delta_{-q\overline{m'}+q'\overline{m} \equiv n \mods{qq'}}.
\end{align*}
\normalsize
We let $d= (q,q')$ and notice that looking at the congruence condition above modulo $q$ implies that $q'\overline{m} \equiv n \mod q$, which in turn implies that $d$ divides $n$. We let $q_0:= q/d, q_0'= q'/d$ and $n_0:= n/d$, so that 
$$n_0 \equiv q_0'\overline{m} \mods{q_0}, \, \hbox{and }n_0 \equiv q_0 \overline{m'} \mods{q_0'}.$$
We may thus bound 
\begin{align*}
S_{1,J,-}^\sharp (t,C,L)&\ll\frac{t^{2\sigma+\epsilon}}{t^{3/2}K^{1/2}C^2L^{1/2}} \sum_{\substack{H\leq \frac{C(tK)^{1/2}t^\epsilon}{L}\\ H \hbox{ Dyadic}}}H^{-1/2} \sum_{C<q,q'\leq 2C} \sum_{\frac{H}{d} < n_0 \leq \frac{2H}{d}} \\
&\times  \sum_{\substack{1\leq |m|\ll qt^\epsilon \\ (m,q)=1}} \sum_{\substack{1\leq |m'|\ll q't^\epsilon \\ (m',q')=1}} \delta_{q_0'\overline{m} \equiv n_0 \mods{q_0}} \delta_{q_0 \overline{m'} \equiv n_0 \mods{q_0'}}\\
&\ll \frac{t^{2\sigma+\epsilon}}{t^{3/2}K^{1/2}C^2L^{1/2}} \sum_{\substack{H\leq \frac{C(tK)^{1/2}t^\epsilon}{L}\\ H \hbox{ Dyadic}}}H^{-1/2} \sum_{C<q,q'\leq 2C} \sum_{\frac{H}{d} < n_0 \leq \frac{2H}{d}} t^\epsilon d^2\\
&\ll \frac{t^{2\sigma +\epsilon}(tK)^{1/4}}{t^{3/2} K^{1/2} C^{3/2} L} \sum_{C<q,q'\leq 2C} d\\
&\ll \frac{t^{2\sigma +\epsilon}(tK)^{1/4}}{t^{3/2} K^{1/2} C^{3/2} L}  \sum_{d\leq 2C} \sum_{\frac{C}{d}\leq q_0,q_0'\leq \frac{2C}{d}} d\\
&\ll \frac{t^{2\sigma+\epsilon}}{tK^{1/2}L}.
\end{align*}
We conclude that 
$$S_{1,J,-}(t,C,L) \ll t^{2\sigma + \epsilon} \left( \frac{1}{t^{3/2}K^{1/2}} +  \frac{1}{tK^{1/2}L}\right).$$
The same bound holds for $S_{1,J,+}(t,C,L)$, via the same analysis, so that
\begin{align*}
S_{1,J}(t,C) &\ll Kt^{\sigma+1/2+\epsilon} \sum_{\substack{1\leq L \ll Kt^\epsilon \\ L \hbox{ Dyadic}}} \left(\frac{L^{1/2}}{t^{3/4}K^{1/4}}+ \frac{1}{t^{1/2}K^{1/4}} \right)\\
&\ll t^{\sigma+\epsilon} \left(\frac{K^{5/4}}{t^{1/4}} + K^{3/4} \right).
\end{align*}
The same bound holds for all values of $J$. Since there are $O(\log t)$ many terms, we can sum over them without worsening the bound, and so the same bound holds for $\hat{S}_{1}(t,C):= \sum_J S_{1,J}(t,C)$. Thus the total contribution of $\hat{S}_1(t,C)$ to $S_l^+(t)$ is bounded by
$$\frac{t^{1+\epsilon}}{K} \left( \frac{K^{5/4}}{t^{1/4}} + K^{3/4} \right) \ll t^\epsilon \left( t^{3/4}K^{1/4} + \frac{t}{K^{1/4}}\right).$$
Choosing $K=t^{1/2}$, we obtain
$$S_l^+(t) \ll t^{1-1/8 +\epsilon}.$$

\section{Examples}\label{sectex}
In this section, we study some examples of analytic trace functions to motivate the analogy with Frobenius trace functions studied in \cite{FKM}. The analog of Kloosterman sums is given in the following example.
\begin{prop}
Let 
$$F_{it}(x) := t^{1/2} \Gamma\left(\frac{1}{2}+it\right) J_{it}(x)$$
be the normalized $J$-Bessel function of order $t$. Then, $F_{it}$ is an analytic trace function.
\end{prop}

\begin{proof}
By \cite[p. 331]{tablestransf}, the Mellin inversion theorem holds for $F_{it}$ and the Mellin transform is given by
$$M_{F,t}(s) := \int_0^\infty F_{it}(x) x^{s-1}\mathrm{d}x = t^{1/2} \Gamma\left(\frac{1}{2}+it\right) 2^{s-1}  \frac{\Gamma\left(\frac{s+it}{2}\right)}{\Gamma\left(1+\frac{it-s}{2}\right)}, $$
for any $0< \sigma <1 $, where $s=\sigma +i\nu$. We will assume for simplicity that $t\geq 1$, the same argument holding also for negative $t$. In order to understand $M_{F,t}(\sigma +i\nu),$ we differentiate between three cases, using Stirling's formula for some of the Gamma factors. We first note that \begin{equation}\label{g1}
\left|\Gamma\left(\frac{1}{2}+it\right)\right| = \sqrt{2\pi} \exp\left(-\frac{\pi t}{2}\right) (1+O(|t|^{-1}).
\end{equation}

First assume we are in the range where $|t\pm \nu|\geq 1$, then we may apply Stirling's formula to all the Gamma factors, and find that
$$M_{F,t}(s) = t^{1/2}g_{F,t}(s)e(f_{F,t}(s)),$$
where, up to a constant, 
$$g_{F,t}(s) = \exp\left(\frac{\pi}{4} (|t-\nu| -|\nu+t| -2t)\right) |(\nu+t)(t-\nu)|^{\frac{\sigma-1}{2}} (1+O(\max\{ t^{-1}, |t\pm \nu|^{-1}\})),$$ 
and 
$$2\pi f_{F,t}(s) = \frac{\nu+t}{2} \log\left|\frac{\nu + t}{2e}\right| + \frac{\nu-t}{2}\log\left|\frac{t-\nu}{2e}\right| + \nu \log 2.$$
We note that if $\nu \geq -\frac{t}{2}$, then $g_{F,t}(s)$ is negligible. We therefore only focus on the case where $\nu < -\frac{t}{2}$ and verify condition (\ref{condfone}) for $f_{F,t}$. We thus compute 
$$2\pi \frac{\mathrm{d}}{\mathrm{d}\nu}f_{F,t}(s) = \frac{1}{2} \log\left|\frac{t^2- \nu^2}{4e^2}\right|  + 1 +\log 2.$$ 
Since we only consider $\nu \gg t$ by exponential decay of $g_{F,t}$ otherwise, we find that 
$$\log\left|\frac{(t^2-\nu^2)^{1/2}}{x}\right| \ll 1,$$
may only occur if $\nu \asymp t$, for $x\in [t,2t]$.

On the other hand, if we are in the range $|t-\nu| < 1$, then we may not apply Stirling's formula for the Gamma factor in the denominator. However, we will have that $|t+\nu| \gg  t$, and thus by (\ref{g1}) and the exponential decay of Gamma factors, we get that the contribution is negligible. Finally, if we are in the range $|t+\nu| <1$, then the phase of $M_{F,t}(s)$ will be of the form 
$$2\pi \tilde{f}_{F,t}(s) := \frac{\nu-t}{2} \log\left|\frac{t-\nu}{2e}\right| + \nu \log 2,$$
and so 
$$2\pi \frac{\mathrm{d}}{\mathrm{d}\nu} \tilde{f}_{F,t}(s) - \log(x) \gg 1$$
in this region, and is thus negligible by integration by parts. Moreover, looking at $f_{F,t}$, there can be no stationary point in any region such that $\nu = -t +o(t).$  \\
We thus assume from now on that we are in the region where $|t\pm \nu|\gg t$, and $t\ll \nu\leq -t$, and will show that conditions (\ref{condg}), (\ref{condftwo}), (\ref{condf}) and (\ref{condad}) hold for $g_{F,t}(s)$ and $f_{F,t}(s)$. Indeed, in this region, 
$$t^{1/2}g_{F,t}(s)= t^{1/2}|(\nu+t)(\nu-t)|^{\frac{\sigma-1}{2}} (1 + O (t^{-1})) \ll t^{\sigma-1/2},$$
and thus 
$$t^{1/2}\frac{\mathrm{d}^j}{\mathrm{d}\nu^j}g_{F,t}(s) \ll t^{\sigma-1/2-j},$$
for all $j\geq 0$, proving (\ref{condg}). We now compute
$$2\pi \frac{\mathrm{d}^2}{\mathrm{d}\nu^2} f_{F,t}(s) =  \frac{\nu}{(\nu^2-t^2)} \gg \nu^{-1},$$
and thus
$$2\pi \frac{\mathrm{d}^j}{\mathrm{d}\nu^j} f_{F,t}(s) \ll_{j,\epsilon} \nu^{1+\epsilon -j},$$
for all $j\geq 0$, proving (\ref{condftwo}) and (\ref{condf}). Finally we look at
$$2\pi \frac{\mathrm{d}^2}{\mathrm{d}\nu^2} f_{F,t}(s) - \frac{1}{\nu} = \frac{t^2}{\nu (\nu^2-t^2)} \gg \nu^{-1}, $$
proving (\ref{condad}), concluding the proof that $F_{it}$ is an analytic trace function.

\end{proof}
Another interesting example is that of Bessel functions of high rank. These can be thought of as analogs to hyper-Kloosterman sums. We study here higher rank Bessel functions appearing in the Voronoi summation formulas in higher rank (as in \cite{HighBessel}).
\begin{prop}
For any $n\geq 3$, let 
$$J_{n,t} := \frac{t^{\frac{n-1}{2}}}{2\pi in} \int_{(\frac{1}{4})} \Gamma\left(\frac{s-int}{n}\right) \Gamma\left(\frac{s}{n}+ \frac{it}{n-1}\right)^{n-1} e\left(\frac{s}{4}\right) x^{-s}\mathrm{d}s.$$
Then $J_{n,t}$ is an analytic trace function.
\end{prop}
\begin{proof}
Let 
$$M_{J_{n,t}}(s):= \frac{t^{\frac{n-1}{2}}}{n} \Gamma\left(\frac{s-int}{n}\right) \Gamma\left(\frac{s}{n}+\frac{it}{n-1}\right)^{n-1}e\left(\frac{s}{4}\right),$$
with $s=\frac{1}{4}+i\nu$. We assume again for simplicity that $t>1$ and want to show that $M_{J_{n,t}}$ satisfies all the conditions in Definition \ref{anatrace}. As in the case of the Bessel function, we wish to use Stiriling's formula to understand the phase and amplitude of $M_{J_{n,t}}$. Again we distinguish three different cases. First assume we are in the range $|\nu- t|\geq n$ and $|(n-1)\nu + nt|\geq n(n-1).$ We may then apply Stirling's formula to both Gamma factors to obtain
$$M_{J_{n,t}}(s) = e\left(\frac{1}{8}\right)\frac{t^{\frac{n-1}{2}}}{n} g_{J_{n,t}}(s) e(f_{J_{n,t}}(s)),$$
where $g_{J_{n,t}}(s) $ is given by

\begin{align*}
&\exp\left(-\frac{\pi(|\nu-nt| + |(n-1)\nu+nt| +n\nu)}{2n}\right) \left|\frac{\nu-nt}{n}\right|^{\frac{1}{4n}-\frac{1}{2}} \left|\frac{\nu}{n}+\frac{t}{n-1}\right|^{(n-1)(\frac{1}{4n}-\frac{1}{2})}\\
&\times\left(1+O\left((1+|\nu-nt|)^{-1}+\left(1+\left|\frac{\nu}{n}+\frac{t}{n-1}\right|\right)^{-1}\right)\right),
\end{align*}
and 
$$2\pi f_{J_{n,t}}(s) = \frac{(n-1)\nu+nt}{n} \log\left|\frac{\nu}{en}+\frac{t}{e(n-1)}\right| + \frac{\nu-nt}{n}\log\left|\frac{\nu-nt}{ne}\right|.$$
We note that if $\nu \geq -\frac{n}{2(n-1)} t$, then $g_{J_{n,t}}$ is negligible. We therefore only focus on the case where $\nu< -\frac{n}{2(n-1)} t$ and verify condition (\ref{condfone}) for $f_{J_{n,t}}$. We thus compute
$$2\pi \frac{\mathrm{d}}{\mathrm{d}\nu} f_{J{n,t}} (s) = \frac{n-1}{n}\log\left|\frac{\nu}{en}+ \frac{t}{e(n-1)}\right| + \frac{1}{n} \log\left|\frac{\nu-nt}{ne}\right| +1.$$
Since we only consider $\nu \gg t$ by exponential decay of $g_{J_{n,t}}$ otherwise, we find that 
$$\log\left|\left(\frac{(n-1)\nu+nt}{n-1}\right)^{\frac{n-1}{n}} \frac{(\nu-nt)^{\frac{1}{n}}}{xn}\right| \ll 1, $$
may only occur if $\nu \asymp x$, for $x\asymp t$. Moreover, as in the Bessel function case, we see from this that in the two cases where we might not use Stirling's formula for one of the Gamma factors, either $g_{J_{n,t}}$ will be negligible, or the phase cannot vanish and the contribution is also negligible.  \\
We thus assume from now on that we are in the region where $|(n-1)\nu +nt|, |\nu-nt| \gg t$ and $t \ll \nu \leq -\frac{n}{(n-1)}t$, and will show that conditions (\ref{condg}), (\ref{condftwo}), (\ref{condf}) and (\ref{condad}) hold for $g_{J_{n,t}}(s)$ and $f_{J_{n,t}}(s)$. Indeed, in this region, 
$$t^{\frac{n-1}{2}} g_{J_{n,t}}(s) = t^{\frac{n-1}{2}} \left|\frac{\nu-nt}{n}\right|^{\frac{1}{4n}-\frac{1}{2}} \left|\frac{\nu}{n}+\frac{t}{n-1}\right|^{(n-1)(\frac{1}{4n}-\frac{1}{2})} (1+O(t^{-1}))\ll t^{\frac{1}{4}-\frac{1}{2}}, $$
and thus 
$$t^{\frac{n-1}{2}} \frac{\mathrm{d}^j}{\mathrm{d}\nu^j} g_{J_{n,t}}(s) \ll t^{\frac{1}{4}-\frac{1}{2}-j},$$
for all $j\geq 0$, proving (\ref{condg}). We now compute 
$$2\pi \frac{\mathrm{d}^2}{\mathrm{d}\nu^2} f_{J_{n,t}}(s) = \frac{(n-1)\nu + nt(2-n)}{(\nu-nt)((n-1)\nu+nt)} \gg \nu^{-1}, $$
since $\nu<0$, and thus 
$$2\pi \frac{\mathrm{d}^j}{\mathrm{d}\nu^j} f_{J_{n,t}}(s) \ll_{j,\epsilon} \nu^{1+\epsilon -j},$$
for all $j\geq 0$, proving (\ref{condftwo}) and (\ref{condf}). Finally, we look at 
$$2\pi \frac{\mathrm{d}^2}{\mathrm{d}\nu^2} f_{J_{n,t}}(s) - \frac{1}{\nu} = \frac{nt^2}{\nu(\nu-nt)((n-1)\nu+nt)} \gg \nu^{-1},$$
proving (\ref{condad}), concluding the proof that $J_{n,t}$ is an analytic trace function.
\end{proof}
We end this section with an example motivating condition (\ref{condad}). Namely, we study $e(x)$ in the range $x\in [t,2t]$ and show that it satisfies all the conditions to be an analytic trace function, besides (\ref{condad}). By Mellin inversion, we thus have
\begin{align*}
V\left(\frac{x}{t}\right) e(x) &=\frac{1}{2\pi} \int_\R t^{i\nu} V^\dagger(-t,i\nu) x^{-i\nu}\mathrm{d}\nu\\
&:= \frac{1}{2\pi} \int_\R M_{e,t}(i\nu) x^{-i\nu}\mathrm{d}\nu, 
\end{align*}
where
$$M_{e,t}(i\nu) = t^{i\nu} V^\dagger(-t,i\nu).$$
We first note that by Lemma \ref{Munshilemma}, we may assume that $\nu \asymp t$, for otherwise $V^\dagger(-t,i\nu)$ is negligible. We now use Lemma 5 in \cite{MR3369905} to write in this region 
$$M_{e,t}(i\nu)= g_{e,t}(i\nu) e(f_{e,t}(i\nu),$$
where, up to a constant,
$$g_{e,t}(i\nu) = \nu^{-1/2} V\left(-\frac{\nu}{2\pi t}\right)(1+O(\nu^{-3/2})),$$
and
$$f_{e,t}(i\nu)= \frac{\nu}{2\pi}\log\left(-\frac{\nu}{2\pi e}\right) .$$
One now verifies that 
$$g_{e,t}^{(j)}(i\nu) \ll_j \nu^{-1/2-j},$$
for all $j\geq 0$. We compute 
$$f_{e,t}'(i\nu) = \frac{1}{2\pi} \log\left(-\frac{\nu}{2\pi e}\right) + \frac{1}{2\pi},$$
and 
$$f_{e,t}^{(j)}(i\nu) = \frac{(-1)^{j}}{2\pi \nu^{j-1}},$$
for $j\geq 2$. We thus have that $f_{e,t}$ satisfies (\ref{condftwo}), (\ref{condf}), and the only condition not satisfied is (\ref{condad}). Given that our results should generalise to holomorphic forms as well as Eisenstein series, this example illustrates the necessity of condition (\ref{condad}), since the divisor function, $d(n)$, correlates with additive characters \cite[Theorem 7.15]{MR882550}.

\section{Horocycle twists}
In this section, we prove Theorem \ref{horothm}. We thus let $K_t : \R_{>0} \rightarrow \C$ be an analytic trace function, and $f$ be a Maass form as in the previous sections. Let $[\alpha,\beta] \subset [1,2]$ and $V$ be a smooth compactly supported function in $[\frac{1}{2},\frac{5}{2}]$, such that $x^jV^{j}(x) \ll_j 1$. We study 
\small
\begin{align*}
\int_\alpha^\beta f(x+iy) K_{1/y}\left(\frac{x}{y}\right) V(x) \mathrm{d}x&= \sum_{n\not=0} \frac{\r_f(n)}{ |n|^{1/2}} W_{it_f}(4\pi |n| y) \int_\alpha^\beta K_{1/y}\left(\frac{x}{y}\right)e(nx) V(x) \mathrm{d}x. 
\end{align*}
\normalsize
The proof of the theorem will then follow from the following proposition. 
\begin{prop}\label{ftrsf}
Let $K_t$ be an analytic trace function. Then there exists an analytic trace function, $\tilde{K}_t(x)$, such that the Fourier transform,
$$\hat{K}_t(x) := t^{1/2}\int_1^2 K_t(tu) V(u) e(-xu) \mathrm{d}u,$$
satisfies
$$\hat{K}_t(x) = \tilde{K}_t(x) +O(t^{-1/2}).$$
\end{prop}
\begin{proof}
We have
\begin{align*}
\int_1^2 K_t(tu)V(u)e(-xu) \mathrm{d}u &= \frac{1}{2\pi i} \int_{(\sigma)} M_t(s) \int_1^2 (tu)^{-s} V(u) e(-xu) \mathrm{d}u \mathrm{d}s\\
&= \frac{1}{2\pi i} \int_{(\sigma)} M_t(s)t^{-s} V^\dagger(x,1-s) \mathrm{d}s.
\end{align*}

We note that by the properties of $M_t(s)$, discussed in Section \ref{anak}, it is sufficient to consider $\nu \asymp t$, such that for some $x\in [t,2t]$,
\begin{equation}\label{o1}
f_t'(\sigma+i\nu) - \frac{\log x}{2\pi} = o(1),
\end{equation}
for otherwise by repeated integration by parts, the integral is negligible. By Lemma 5 of \cite{MR3369905}, we may write
$$V^\dagger(x,1-\sigma-i\nu) = \frac{\sqrt{2\pi} e(1/8)}{\sqrt{\nu}} V\left(-\frac{\nu}{2\pi x}\right) \left(-\frac{\nu}{2\pi x}\right)^{1-\sigma} \left(-\frac{\nu}{2\pi ex}\right)^{-i\nu} + O(|\nu|^{-3/2}).$$
We thus have that the main term of $\hat{K}_t(x)$ is 
\begin{align*}
&\frac{e(1/8)t^{1/2-\sigma}}{\sqrt{2\pi} i} \int_{(\sigma)} M_t(\sigma+i\nu) W(\nu) \frac{t^{-i\nu}}{\sqrt{\nu}}   V\left(-\frac{\nu}{2\pi x}\right) \left(-\frac{\nu}{2\pi x}\right)^{1-\sigma} \left(-\frac{\nu}{2\pi ex}\right)^{-i\nu} \mathrm{d}\nu,
\end{align*}
where $W$ is a smooth compactly supported function such that $W^{(j)}(\nu)\ll_j \nu^{-j}$, and supported only whenever (\ref{o1}) holds. We may thus rewrite the main term as
$$\frac{1}{2\pi i} \int_{(1-\sigma)} \tilde{M}_{t,x}(1-\sigma+i\nu) x^{\sigma-1-i\nu} \mathrm{d}\nu,$$
where up to a constant, 
$$\tilde{M}_{t,x}(1-\sigma+i\nu) = t^{1/2-\sigma+i\nu} M_t(\sigma-i\nu) W(-\nu) V\left(\frac{\nu}{2\pi x}\right) \nu^{1/2-\sigma+i\nu}(2\pi e)^{-i\nu}.$$
We write
$$\tilde{M}_{t,x}(1-\sigma +i\nu) = \tilde{g}_{t,x}(1-\sigma+i\nu)e(\tilde{f}_t(1-\sigma+i\nu)),$$
where 
$$\tilde{g}_{t,x}(1-\sigma+i\nu)= t^{1/2-\sigma} W(-\nu)g_t(\sigma-i\nu)V\left(\frac{\nu}{2\pi x}\right) \nu^{1/2-\sigma},$$
and
$$\tilde{f}_t(1-\sigma+i\nu)= \frac{\nu}{2\pi} \log(t\nu) + f_t(\sigma-i\nu).$$
We compute
$$\frac{\mathrm{d}}{\mathrm{d}\nu} \tilde{f}_t(1-\sigma + i\nu)  -\frac{1}{2\pi} \log x= \frac{1}{2\pi} \log\left(\frac{t\nu}{2\pi x}\right) -f_t'(\sigma-i\nu),$$
and note that if $\frac{\nu}{2\pi x} \not \in [\frac{1}{2},\frac{5}{2}],$ then by (\ref{o1}), we have that (\ref{condfone}) holds, so that by repeated integration by parts the integral in that region is negligible. We may therefore write 
$$\int_{(1-\sigma)} \tilde{M}_{t,x}(1-\sigma+i\nu) x^{\sigma-1-i\nu}\mathrm{d}\nu = \int_{(1-\sigma)} \tilde{M}_t(1-\sigma+i\nu) x^{\sigma-1-i\nu}\mathrm{d}\nu + O(t^{-100}),$$
where $\tilde{M}_t(1-\sigma+i\nu) = \tilde{g}_t(1-\sigma+i\nu) e(\tilde{f}(1-\sigma+i\nu)),$ and
$$\tilde{g}_t(1-\sigma+i\nu)= t^{1/2-\sigma} W(-\nu) g_t(1-\sigma+i\nu) \nu^{1/2-\sigma}.$$
In the range $\nu \asymp t$, we have
$$\tilde{g}_t^{(j)} (1-\sigma+i\nu) \ll t^{1/2-\sigma-j},$$
and therefore $\tilde{g}_t$ satisfies condition (\ref{condg}). We moreover have
$$\frac{\mathrm{d}^2}{\mathrm{d}\nu^2} \tilde{f}_t(1-\sigma+i\nu) = \frac{1}{2\pi \nu} + f_t''(\sigma-i\nu) \gg \nu^{-1},$$
by (\ref{condad}) and thus (\ref{condftwo}) is satisfied for $\tilde{f}_t$. Moreover, by direct computation, we see that since (\ref{condf}) holds for $f_t$, it also holds for $\tilde{f}_t$. By (\ref{condftwo}), we have
$$\tilde{f}_t''(1-\sigma+i\nu) - \frac{1}{2\pi \nu} = f_t''(\sigma-i\nu) \gg \nu^{-1},$$
so that (\ref{condad}) holds for $\tilde{f}_t$.
\end{proof}
We deduce Theorem \ref{horothm} from Proposition \ref{ftrsf}. We first note that the exponential decay of $W_{it_f}$ restricts $n$ to the range $|n|\ll y^{-1}.$ Keeping in mind that the Fourier transform is negligible unless $n\asymp y^{-1},$ we only need to show that
$$\frac{1}{\beta-\alpha}\sum_{n\asymp y^{-1}} \frac{\r_f(n)}{y^{1/2}|n|^{1/2}} y^{1/2} \int_\alpha^\beta K_{1/y}\left(\frac{x}{y}\right) e(nx) V(x)\mathrm{d}x \rightarrow 0,$$
as $y\rightarrow 0$. However, by Fourier inversion, we have
\begin{align*}
y^{-1/2}\int_\alpha^\beta K_{1/y}\left(\frac{x}{y}\right) e(nx)V(x)\mathrm{d}x &= \int_\alpha^\beta \int_\R \hat{K}_{1/y}(z) e(zx)e(nx) \mathrm{d}z\mathrm{d}x\\
&= \int_\R \hat{K}_{1/y}(z+n) \int_\alpha^\beta e(zx) \mathrm{d}x\mathrm{d}z\\
&= \frac{1}{2\pi i}\int_\R \hat{K}_{1/y}(z+n) \frac{e(\beta z)-e(\alpha z)}{z} \mathrm{d}z.
\end{align*}
Now by Proposition \ref{ftrsf} and the properties of analytic trace functions, we must have $z+n \asymp y^{-1}$, for otherwise $\tilde{K}_{1/y}(z+n)$ is negligible. We may thus apply Theorem \ref{thm1} to conclude that
\begin{align*}
\frac{1}{\beta-\alpha}\sum_{n\asymp y^{-1}} \frac{\r_f(n)}{y^{1/2}|n|^{1/2}} y^{1/2} \int_\alpha^\beta K_{1/y}\left(\frac{x}{y}\right) e(nx) V(x)\mathrm{d}x \ll \frac{y^{1/8 -\epsilon}}{\beta-\alpha}, 
\end{align*}
proving Theorem \ref{horothm}.

\bibliography{refs}
\bibliographystyle{plain}

\end{document}